\documentclass[a4paper, 10pt,leqno]{amsart}
\date{\today}
\usepackage[dvipdfmx]{graphicx}
\usepackage{latexsym}
\usepackage{amssymb}
\usepackage{amsmath}
\usepackage{amscd}
\usepackage{amsthm}
\usepackage{amsfonts}
\usepackage{mathrsfs}
\usepackage{enumerate}
\usepackage{enumitem}
\usepackage{bm}
\usepackage[usenames]{color}
\usepackage[active]{srcltx}
\usepackage{comment}
\usepackage[bbgreekl]{mathbbol}
\usepackage{here}
\usepackage[subrefformat=parens]{subcaption}
  \DeclareSymbolFontAlphabet{\mathbb}{AMSb}
  \DeclareSymbolFontAlphabet{\mathbbl}{bbold}
  \DeclareMathSymbol{\bbepsilon}{\mathord}{bbold}{"0F}
\title{The total absolute curvature of closed curves with singularities}
\author[A.~Honda]{Atsufumi Honda}
\address{%
   Department of Applied Mathematics, 
   Faculty of Engineering, Yokohama National University, %\endgraf
   Hodogaya, Yokohama 240-8501, Japan
}
\email{honda-atsufumi-kp@ynu.ac.jp}
\author[C.~Tanaka]{Chisa Tanaka}
\address{%
   College of Engineering Science, 
   Yokohama National University, %\endgraf
   Hodogaya, Yokohama 240-8501, Japan
}
\curraddr{%
   NTT Data Frontier Corporation,
   Konan, Tokyo 108-0075, Japan
}
\email{rgt10tkd@gmail.com}
\author[Y.~Yamauchi]{Yuta Yamauchi}
\address{%
   Graduate School of Engineering Science, 
   Yokohama National University, %\endgraf
   Hodogaya, Yokohama 240-8501, Japan
}
\email{yamauchi-yuta-hj@ynu.jp}
\thanks{
The first author is supported by 
JSPS KAKENHI Grant Numbers 19K14526, 20H01801 
from Japan Society for the Promotion of Science.}
\usepackage{amsthm}
\theoremstyle{plain}
 \newtheorem{theorem}{Theorem}[section]
 \newtheorem{introtheorem}{Theorem}

 \newtheorem{proposition}[theorem]{Proposition}
 \newtheorem{fact}[theorem]{Fact}
 \newtheorem*{fact*}{Fact}
 \newtheorem{lemma}[theorem]{Lemma}
 \newtheorem{corollary}[theorem]{Corollary}
 \theoremstyle{remark}
 \newtheorem{definition}[theorem]{Definition}
 
 \newtheorem*{acknowledgements}{Acknowledgements}
 \newtheorem{example}[theorem]{Example}
\numberwithin{equation}{section}

 \makeatletter
    
    \@addtoreset{equation}{section}
  \makeatother

\newcommand{\Z}{\mathbb{Z}}%{\boldsymbol{Z}}
\newcommand{\R}{\mathbb{R}}%{\boldsymbol{R}}
\newcommand{\vect}[1]{\boldsymbol{#1}}

\newcommand{\sgn}{\operatorname{sgn}}

    \newcommand{\vt}{{\vect{e}}}

\newcommand{\ga}{\gamma}
\newcommand{\Ga}{\Gamma}

\newcommand{\M}{N}

\newcommand{\Reg}{\operatorname{Reg}}
\begin{document}
\begin{abstract}
In this paper, we give a generalization of 
Fenchel's theorem for closed curves as frontals in Euclidean space $\R^n$.
We prove that, for a non-co-orientable closed frontal in $\R^n$,
its total absolute curvature is greater than or equal to $\pi$.
It is equal to $\pi$ if and only if 
the curve is a planar locally $L$-convex closed frontal
whose rotation index is $1/2$ or $-1/2$.
Furthermore, 
if the equality holds and if every singular point is a cusp,
then the number $\M$ of cusps 
is an odd integer greater than or equal to $3$,
and 
$\M=3$ holds if and only if 
the curve is simple.
\end{abstract}
\maketitle

\section{Introduction}
We fix an integer $n$ greater than $1$.
Let $\gamma(s)$ $(s\in [0,L])$ be an
arclength parametrization of a closed regular curve 
in Euclidean $n$-space $\R^n$,
where $L$ is the length of $\gamma(s)$.
Denote by $k(s)\,(\geq0)$ the curvature function of 
$\gamma(s)$.
Then,
$$
  K(\gamma)=\int_{0}^{L} k(s) \,ds
  \qquad
  \left(k=\left\| \frac{d^2\gamma}{ds^2}\right\| \right)
$$
is called the {\it total absolute curvature}.
The following holds:

\medskip

\noindent
{\bf Fenchel's theorem}
(\cite{Fenchel1929, Fenchel1951, Borsuk, Rutishauser-Samelson}){\bf .}~
%\label{fact:Fenchel}
{\it 
The total absolute curvature 
of a closed regular curve in $\R^n$ is 
greater than or equal to $2\pi$.
It is equal to $2\pi$ if and only if 
the curve is an oval.}

\medskip

Here, an oval is a locally convex simple closed plane curve.
We remark that 
a locally convex closed plane curve 
is simple if and only if its rotation index is equal to $1$ or $-1$.
%
%Fact \ref{fact:Fenchel} is known as a Fenchel's theorem.
So far, several generalizations of Fenchel's theorem have been obtained:
knots in $\R^3$ \cite{Fary, Milnor};
curves in a non-positively curved Riemannian manifolds
\cite{Tsukamoto, Brickell-Hsiung};
curves in a sphere \cite{Teufel1986, Teufel1992};
open curves in $\R^n$ \cite{EIS};
curves in the Lorentz-Minkowski space \cite{Borisenko-Tenenblat};
curves in $\operatorname{CAT}(\kappa)$ space \cite{Sama-Ae-Phon-on}.
See also \cite{Arnol'd}.
Furthermore,
Chern-Lashof \cite{Chern-Lashof, Chern-Lashof2}
generalized Fenchel's theorem
to closed submanifolds in $\R^n$.
The lower bound of the total absolute curvature 
is related to topological invariant of the submanifold.
Therefore, Fenchel's theorem is one of 
the main results in global differential geometry.
In this paper, we extend Fenchel's theorem 
for curves with {\it singular points}.
%

%%%%%%%%%%%%%%%%%%%%%%%%%%%%%%%%%%
%%%%%%%%%%%%%%%%%%%%%%%%%%%%%%%%%%
\subsection{Statement of results}
\label{sec:results}

For a smooth map $\gamma:I \to \R^n$
defined on a non-empty interval $I$,
a point $c\in I$
is called a {\it regular point\/}
(resp.\ a {\it singular point\/})
of $\gamma$
if the derivative $\gamma'(c)$ does not vanish
(resp.\ does vanish),
where the prime $'$ means $d/dt$.
In this paper, a {\it curve}
is a smooth map 
$$\gamma:I \longrightarrow \R^n$$
whose regular set $\Reg(\gamma)$ is dense in $I$.
A curve $\gamma:I \to \R^n$
is called a {\it frontal}
if there exists a smooth map 
$\vt:I \to S^{n-1}$
such that 
$\ga' (t)$ and $\vt(t)$
are linearly dependent for each $t\in I$,
where $S^{n-1}$ is the unit sphere in $\R^n$.
Such an $\vt(t)$ is said to be a {\it unit tangent vector field\/}
along $\gamma$.
Moreover, 
a frontal $\gamma$ is said to be a {\it wave front\/}
(or a {\it front\/}, for short),
if $L=(\gamma,\vt)$ is a regular curve in 
$\R^n\times S^{n-1}$.

On the regular set $\Reg(\gamma)$
of a frontal $\gamma:I\to \R^n$,
the {\it curvature function} is defined as
$$
k(t)
=\frac{\sqrt{\|\gamma'\|^2\|\gamma''\|^2 - (\gamma'\cdot \gamma'')^2}
}{\|\gamma'\|^3}.
$$
Here,  
we denote by $\|\vect{a}\|$
the norm $\|\vect{a}\|:=\sqrt{\vect{a}\cdot\vect{a}}$ 
of a vector $\vect{a}\in \R^n$,
and $\vect{a} \cdot \vect{b}$
is the canonical Euclidean inner product
of $\vect{a},\vect{b} \in \R^n$.
Since $\vect{e}(t)=\pm \gamma'(t)/\|\gamma'(t)\|$ 
holds on $\Reg(\gamma)$,
the curvature function $k(t)$ is written as
$$
%k(t)=\frac{\|\vect{e}'(t)\|}{\|\gamma'(t)\|}.
k(t)={\|\vect{e}'(t)\|}/{\|\gamma'(t)\|}.
$$
In general, the curvature function $k$
of a frontal may diverge at a singular point.
However, 
\begin{equation}\label{eq:kappa}
  k\,ds=\|\vect{e}'(t)\|\,dt
\end{equation}
is a bounded continuous $1$-form even at a singular point,
where 
$
 ds=\|\gamma'(t)\|\,dt
$
is the arclength measure.
We call $k\,ds$
the {\it curvature measure}.

If a frontal $\gamma : \R\to \R^n$ is a periodic map,
$\gamma$ is called a {\it closed frontal}.
By multiplying the parameter by a constant, 
we may assume that $\gamma$
is $2\pi$-periodic.
Then, the domain of definition of $\gamma$
is regarded as
$S^1=\R/2\pi \Z$.
For each $a\in \R$,
we represent the $2\pi$-periodic closed curve
$\gamma(t)$
by the restriction $\gamma|_{[a,a+2\pi]}$.

If $\vt$ is also a $2\pi$-periodic map,
that is, if 
$$\vt(t+2\pi) = \vt(t) \qquad (t\in \R),$$
then $\gamma$ is said to be {\it co-orientable}.
Then, the domain of definition of $\vt$ is regarded as $S^1$.
On the other hand, 
if 
$$\vt(t+2\pi) = -\vt(t) \qquad (t\in \R),$$
then $\gamma$ is called {\it non-co-orientable}.
In this case, $\vt$ is a $4\pi$-periodic map,
and hence, 
the domain of definition of $\vt$ is regarded as 
$\widetilde{S}^1=\R/4\pi\Z$.

For a closed frontal $\gamma:S^1\to \R^n$,
as the curvature measure is a continuous $1$-form
on $S^1$,
\begin{equation}\label{eq:K}
  K(\gamma)
  =\int_{S^1} k \,ds
%  =\int_{S^1} \|\vect{e}'(t)\| \,dt
\end{equation}
is a bounded non-negative number.
We call $K(\gamma)$ the {\it total absolute curvature}.

We remark that, in general, 
the total absolute curvature $K(\gamma)$ 
does not possesses a non-trivial lower bound.
In fact, the line segment 
$\gamma(t) =(\cos t,0)$ is a closed frontal
with vanishing total absolute curvature $K(\gamma)=0$.
Similar examples exist even when restricted to wave fronts,
see Example \ref{ex:megata}.

\begin{figure}[htb]
\centering
 \begin{tabular}{c}
\resizebox{5cm}{!}{\includegraphics{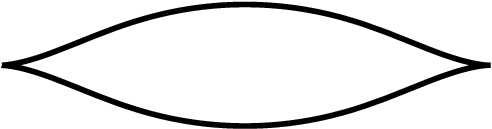}}
\end{tabular}  
\caption{A co-orientable closed front 
having arbitrary small total absolute curvature
(Example \ref{ex:megata}).
}
\label{fig:mgt}
\end{figure}

Hence, in this paper, 
we deal with non-co-orientable closed frontals in $\R^n$.
More precisely,
we prove the following Fenchel-type theorem.

\begin{introtheorem}
\label{thm:intro2}
The total absolute curvature of a non-co-orientable closed frontal 
in $\R^n$ is greater than or equal to $\pi$.
It is equal to $\pi$ if and only if 
the curve is a planar locally $L$-convex frontal
whose rotation index is equal to $1/2$ or $-1/2$.
\end{introtheorem}

Here, the definitions of the local $L$-convexity
and the rotation index are given in Section \ref{sec:thm-A}.
In the case of regular curves,
Fenchel's theorem says that 
every closed regular curve having minimum total absolute curvature
$2\pi$
must be a simple closed curve.
However, in the case of frontals,
there exist non-simple examples 
having minimum total absolute curvature
$K(\gamma)=\pi$, see Figure \ref{fig:curve}.

\begin{figure}[htb]
\centering
 \begin{tabular}{c@{\hspace{4mm}}c@{\hspace{4mm}}c@{\hspace{4mm}}c}
 %{ccc}
  \resizebox{3cm}{!}{\includegraphics{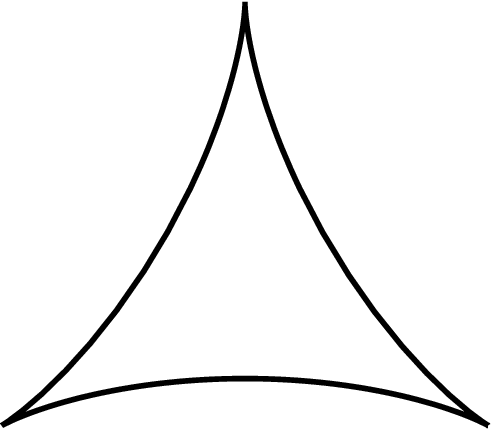}} &
  \resizebox{3cm}{!}{\includegraphics{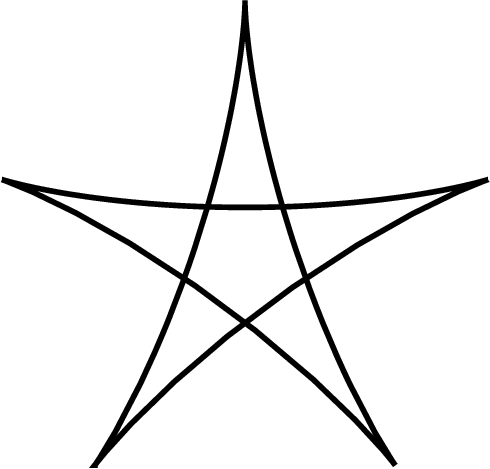}} & 
  \resizebox{3cm}{!}{\includegraphics{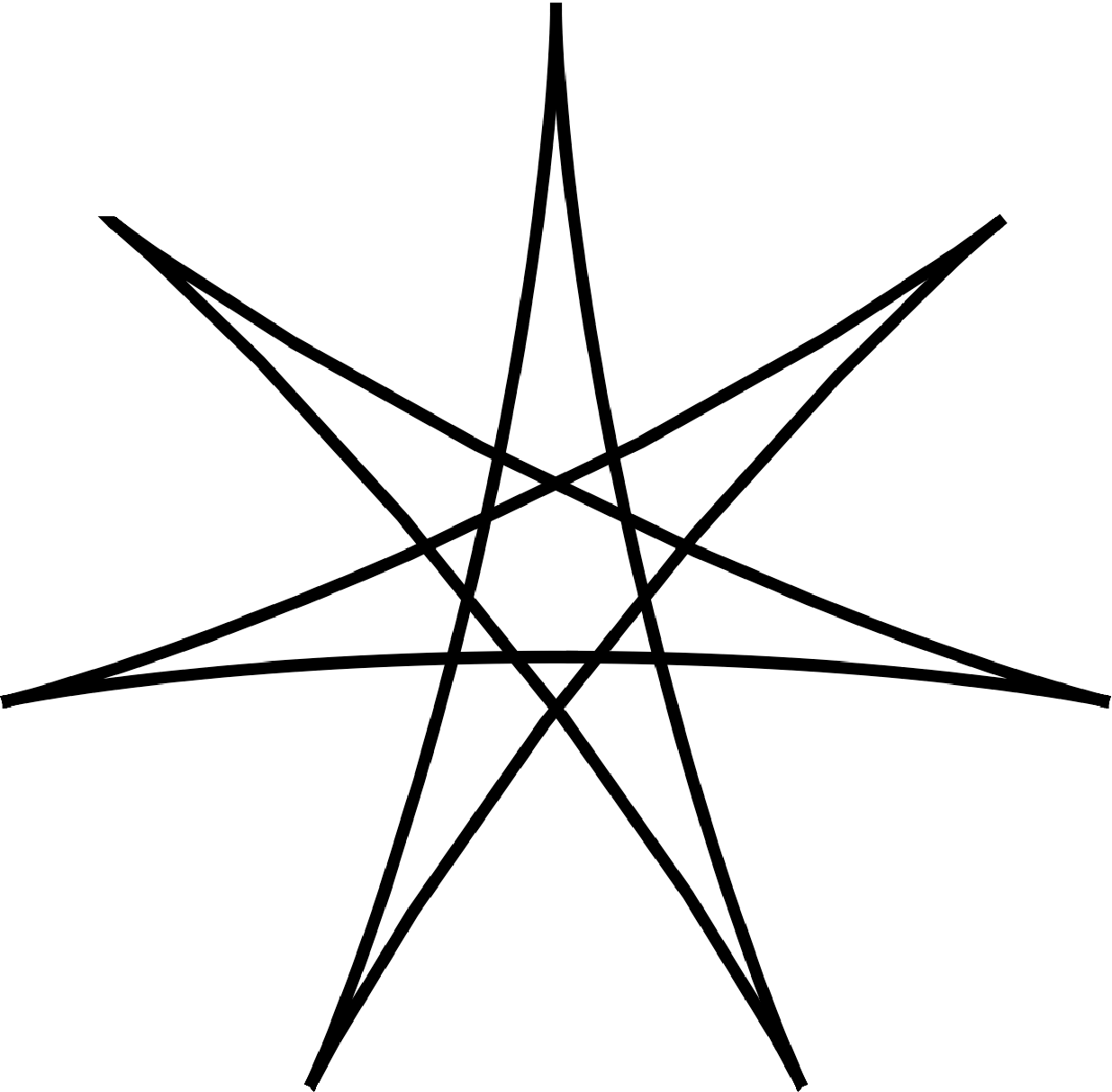}} \\
  {\footnotesize (a)} &
  {\footnotesize (b)} &
  {\footnotesize (c)} 
 \end{tabular}  
 \caption{
 Closed non-co-orientable fronts
 with total absolute curvature $K(\gamma)=\pi$ in $\R^2$.
 All curves (a), (b) and (c) are given by hypocycloids 
 (Example \ref{ex:hypo}).
 Although the leftmost (a) is a simple closed curve,
 the others, (b) and (c), are not simple. 
 All the singular points of these curves are cusp. 
 }
 \label{fig:curve}
\end{figure} 

Thus, a natural question to consider 
is when a non-co-orientable closed frontal 
with minimum total absolute curvature 
becomes a simple closed curve.
In this paper, 
we solve this question 
when the singularities are all {\it cusps}.
Here, for a smooth map 
$\gamma : I \to \R^2$,
a singular point $c\in I$ of $\gamma$ is called {\it cusp}
if the map-germ $\gamma : (I,c) \to (\R^2,\gamma(c))$
is right-left equivalent to the map-germ $t\mapsto (t^2,t^3)$.
It is known that the cusp singularity
is stable among the singularities of plane curves.
We have the following. 

\begin{introtheorem}
\label{thm:intro3}
Let $\gamma : S^1\to \R^2$ be a
non-co-orientable closed front
with minimum total absolute curvature $K(\gamma)=\pi$.
Suppose that every singular point is a cusp.
Then the number $\M$ of cusps is 
an odd integer 
satisfying $\M\geq3$.
Moreover, 
$\M=3$ if and only if $\gamma$ is simple.
\end{introtheorem}

On the other hand,
for a given convex closed curve,
its focal set (i.e., the image of the caustic)
admits singular points.
Gounai-Umehara \cite{Gounai-Umehara}
classified the diffeomorphic type of focal sets 
of convex curves which admit at most four cusps
under a generic assumption.
Such the caustics are 
wave fronts having no inflection points.
Theorem \ref{thm:intro3} can be regarded as 
a generalization of Gounai-Umehara's classification theorem
in the case that the focal sets have exactly three singular points.

%%%%%%%%%%%%%%%%%%%%%%%%%%%%%%%%%%
%%%%%%%%%%%%%%%%%%%%%%%%%%%%%%%%%%
\subsection{Organization of the paper}
\label{sec:organization}

This paper is organized as follows.
We prove Theorem \ref{thm:intro2} in Section \ref{sec:thm-A},
and Theorem \ref{thm:intro3} in Section \ref{sec:thm-B}.
In Section \ref{sec:examples}, we exhibit some examples. 
In particular, 
we show an example of a co-orientable closed wave front 
with arbitrarily small total absolute curvature
(Example \ref{ex:megata}),
and 
non-co-orientable closed hypocycloids
with total absolute curvature $\pi$ 
(Example \ref{ex:hypo}),
as shown in Figure \ref{fig:curve}.

%%%%%%%%%%%%%%%%%%%%%%%%%%%%%%%%%%
%%%%%%%%%%%%%%%%%%%%%%%%%%%%%%%%%%
\section{Proof of Theorem \ref{thm:intro2}}
\label{sec:thm-A}

In this section, we give a proof of Theorem \ref{thm:intro2}.
First, we prove that 
for a non-co-orientable closed frontal 
$\gamma : S^1\to \R^n$,
the total absolute curvature $K(\gamma)$
satisfies $K(\gamma)\geq \pi$,
and that $K(\gamma)= \pi$ implies $\gamma(t)$ must be planar
(Proposition \ref{prop:A-1}).
Next, 
after introducing the definitions of the local $L$-convexity
and the rotation index,
we prove that 
a non-co-orientable closed frontal 
$\gamma : S^1\to \R^2$
with total absolute curvature $K(\gamma)=\pi$
must be locally $L$-convex and 
the rotation index is $\pm1/2$,
and vice versa
(Proposition \ref{thm:thm1}).

\subsection{Length of tangent indicatrix}
By \eqref{eq:kappa} and \eqref{eq:K}, we have the following.

\begin{lemma}\label{lem:t-length}
For a closed frontal $\ga: [0,2\pi] \to \R^n$, 
the length
$$\mathcal{L}(\vt) = \int_0^{2\pi} \left\| \vt'(t) \right\| dt$$
of $\vt : [0,2\pi] \to S^{n-1}$ as a spherical curve in $S^{n-1}$ 
coincides with the total absolute curvature $K(\gamma)$ 
of the frontal $\gamma$.
\end{lemma}

The non-co-orientability yields the following.

\begin{lemma}\label{lem:inter}
Let $\ga : [0,2\pi] \to \R^n$ be a non-co-orientable closed frontal
with a unit tangent vector field $\vt:[0,2\pi] \to S^{n-1}$.
Then, every great hypersphere of $S^{n-1}$ intersects 
the image $\vt([0,2\pi])$.
\end{lemma}

\begin{proof}
Let $G$ be a great hypersphere of $S^{n-1}$.
Then there exists a vector $\xi \in S^{n-1}$
such that 
$
G = \{ \vect{x} \in S^{n-1} \,;\, \vect{x}\cdot \xi = 0 \}.
$
Hence, 
it suffices to prove that 
there exists $t_0 \in [0,2 \pi]$ such that
$\vt(t_0) \cdot  \xi = 0$.
We set $f : [0,2 \pi] \to \R$ by
$
f(t) = \vt(t) \cdot \xi.
$
Since $\ga$ is non-co-orientable,
$\vt(0) = - \vt(2 \pi)$ holds
and hence,
we have
$
f(0) = - f(2 \pi).
$
If $f(0)=0$, 
then the claim is true, 
so let us assume that $f(0) \neq 0$.
Since $f(0)$ and $f(2 \pi)$ have the opposite sign,
the intermediate value theorem yields that
there exists $t_0 \in [0,2 \pi]$ such that $f(t_0) = 0$,
which proves the assertion.
\end{proof}

We use the following fact called the {\it Rutishauser-Samelson lemma}
\cite{Rutishauser-Samelson}.
This lemma is also introduced in an expository article 
written by S.S.~Chern \cite{Chern}.

\begin{fact}[\cite{Rutishauser-Samelson}, see also \cite{Chern}]
\label{fact:sphere}
Let $\vect{c} : S^1 \to S^{n-1}$ be a $C^1$ closed curve 
of length $\mathcal{L}(\vect{c})$.
We denote by $\Ga$ the image $\vect{c}(S^1)$.
Suppose that, 
for any great hypersphere $G$ of $S^{n-1}$,
the intersection $\Ga \cap G$ is non-empty. 
Then $\mathcal{L}(\vect{c}) \geq 2 \pi$ holds.
Moreover, if $\mathcal{L}(\vect{c}) = 2 \pi$,
then $\Ga$ consists of two half-arcs of great circles.
\end{fact}

Now we have the following.

\begin{proposition}\label{prop:A-1}
Let $\ga : S^1 \to \R^n$ be a non-co-orientable closed frontal.
Then, $K(\ga) \geq \pi$ holds.
Moreover, if the equality holds,
then $\ga$ is planar.
\end{proposition}

\begin{proof}
Let $\tilde{\ga}$ and $\tilde{\vt}$ be the extensions of $\ga$ and $\vt$
so that their domains of definition are $[0, 4 \pi]$.
Then, the domain of definition of $\tilde{\vt}$
is regarded as $\tilde{S}^1=\R/4\pi\Z$,
and the image $\tilde{\Ga}=\tilde{\vt}(\tilde{S}^1)$ defines a closed curve in $S^{n-1}$.
From Lemma \ref{lem:inter}, $\tilde{\Ga}$ intersects with any great hypersphere of $S^{n-1}$.
Therefore, Fact \ref{fact:sphere} yields
$
K(\tilde{\ga}) \geq 2 \pi.
$
Since $K(\tilde{\ga}) = 2 K(\ga)$,
we have $K(\ga) \geq \pi.$
The equality holds if and only if $\tilde{\Ga}$ consists of two half-arcs 
of great circles, that is, $\Ga=\vt([0,2\pi])$ is a half-arc of a great circle.
Therefore, the image of $\ga(t)$ is a subset of 
a $2$-dimensional subspace
of $\R^n$,
which proves the assertion.
\end{proof}

%%%%%%%%%%%%%%%%%%%%%%%%%%%%%%%%%%
%%%%%%%%%%%%%%%%%%%%%%%%%%%%%%%%%%
\subsection{Local $L$-convexity of planar frontals}
On the regular set $\Reg(\gamma)$
of a frontal $\gamma:I\to \R^2$,
we set
$$
\kappa(t)
=\frac{1}{\|\gamma'\|^3}\det(\gamma',\gamma'').
$$
To distinguish between $\kappa(t)$ and $k(t)$,
we call $\kappa(t)$ the {\it oriented curvature function} of $\gamma(t)$.
Since $\vect{e}(t)=\pm \gamma'(t)/\|\gamma'(t)\|$ 
holds on $\Reg(\gamma)$,
the oriented curvature function $\kappa(t)$ is written as
$
\kappa(t)=\tilde{\kappa}(t)/\|\gamma'(t)\|,
$
where 
we set
\begin{equation}\label{eq:tilde-kappa}
  \tilde{\kappa}(t) = \det(\vt(t),\vt'(t))\qquad
  (\vt'=d\vt/dt).
\end{equation}
Letting 
$
 ds=\|\gamma'(t)\|\,dt
$
is the arclength measure of $\gamma$,
we have that 
\begin{equation}\label{eq:ori-kappa}
  \kappa\,ds=\tilde{\kappa}\,dt
\end{equation}
is a bounded smooth $1$-form even at a singular point.
We call $\kappa\,ds$
the {\it oriented curvature measure}.

\begin{definition}
A frontal $\gamma:I\to \R^2$ is called 
{\it locally $L$-convex}, if either
$\tilde{\kappa}\geq0$ or $\tilde{\kappa}\leq0$ 
holds on $I$,
where $\tilde{\kappa}$ is the function defined by 
\eqref{eq:tilde-kappa}.
\end{definition}

Since the oriented curvature function $\kappa$ is written as
$\kappa=\tilde{\kappa}/\|\gamma'\|$ on the regular set $\Reg(\gamma)$,
the local $L$-convexity condition is a generalization 
of the local convexity of regular curves.
We remark that 
Li-Wang \cite{Li-Wang}
defined {\it $\ell$-convex curves},
cf.\ \cite{Fukunaga-Takahahi}.
An $\ell$-convex curve is a closed wave front in $\R^2$
such that $\tilde{\kappa}$ has no zeros,
and hence, $\ell$-convex curves are locally $L$-convex.

\medskip

Let $\gamma : S^1\to \R^2$ 
be a non-co-orientable closed frontal
with a unit tangent vector field  $\vect{e} : [0,2\pi]\to S^1$.
We set a smooth function $\theta(t)$ as
\begin{equation}\label{eq:theta}
\vt(t)=(\cos \theta(t), \sin \theta(t)).
\end{equation}
We call $\theta(t)$ the {\it angle function}.
Since $\vect{e} : [0,2\pi]\to S^1$ satisfies
$\vect{e}(2\pi)=-\vect{e}(0)$,
there exists an integer $m\in \Z$
such that $\theta(2\pi)-\theta(0)=(2m+1)\pi$ holds.
We call the half-integer
$$
  \operatorname{ind}_{\gamma}= m+\frac1{2}
$$
the {\it rotation index} of $\gamma(t)$.
We remark that $2\pi \operatorname{ind}_{\gamma}=\theta(2\pi)-\theta(0)$.

\begin{lemma}\label{lem:theta-prime}
For a frontal $\gamma : I \to \R^2$,
we set a smooth function $\theta(t)$ as \eqref{eq:theta}.
Then $\theta'(t)=\tilde{\kappa}(t)$ holds.
\end{lemma}

\begin{proof}
Substituting \eqref{eq:theta} into
\eqref{eq:tilde-kappa},
we obtain the desired result.
\end{proof}

\begin{proposition}\label{thm:thm1}
Let $\gamma : S^1 \to \R^2$ be a non-co-orientable closed frontal.
Then, the total absolute curvature $K(\gamma)$
is greater than or equal to $\pi$.
It is equal to $\pi$ if and only if 
$\gamma$ is an $L$-convex frontal 
whose rotation index is equal to $1/2$ or $-1/2$.
\end{proposition}

\begin{proof}
Let $\vect{e} : [0,2\pi] \to S^1$ be a unit tangent vector field along 
$\gamma(t)$.
We set $\theta(t)$ the angle function of $\vect{e}(t)$ 
as in \eqref{eq:theta}.
The rotation index is written as
$\operatorname{ind}_{\gamma}=m+1/2$,
where $m \in \Z$ is an integer.
By the triangle inequality, we have
$$
  K(\gamma) 
  =\int_0^{2\pi} |\theta'(t)|\,dt 
  \geq \left|\int_0^{2\pi} \theta'(t)\,dt\right|
  = |\theta(2\pi)-\theta(0)|
  \geq |2m+1|\pi
  \geq \pi.
$$
The equality holds if and only if
$\theta'(t)$ has constant sign and
$m =0,-1$.
Lemma \ref{lem:theta-prime} yields that
$K(\gamma)=\pi$ holds if and only if
$\gamma(t)$ is locally $L$-convex and 
$\operatorname{ind}_{\gamma}=\pm1/2$.
\end{proof}

\begin{proof}[Proof of Theorem \ref{thm:intro2}]
By Propositions \ref{prop:A-1} and \ref{thm:thm1}, 
we have the assertion.
\end{proof}

%%%%%%%%%%%%%%%%%%%%%%%%%%%%%%%%%%
%%%%%%%%%%%%%%%%%%%%%%%%%%%%%%%%%%
%%%%%%%%%%%%%%%%%%%%%%%%%%%%%%%%%%
%%%%%%%%%%%%%%%%%%%%%%%%%%%%%%%%%%
\section{Proof of Theorem \ref{thm:intro3}}
\label{sec:thm-B}

In this section, we prove Theorem \ref{thm:intro3}.
First, in subsection \ref{sec:cusp},
we discuss the sign function of frontals having cusp singularities.
Next, in subsection \ref{sec:angle},
several estimates of the total absolute curvature are given.
Finally, in subsection \ref{sec:proof-B},
we give a proof of Theorem \ref{thm:intro3}.

%%%%%%%%%%%%%%%%%%%%%%%%%%%%%%%%%%
%%%%%%%%%%%%%%%%%%%%%%%%%%%%%%%%%%
\subsection{Sign function}
\label{sec:cusp}

Let $\ga : I\to \R^2$ be a smooth curve.
Then a singular point $t=c$ is called {\it cusp},
if there exist a coordinate change
$t=t(s)$ and a local diffeomorphism 
$\Phi$ of $\R^2$ such that 
$t(0)=c$,  $\Phi(\gamma(c))=(0,0)$
and
$
\Phi \circ \gamma(t(s))=(s^2,s^3)
$
hold.

It is well-known that 
a plane curve $\ga(t)$ has cusp at $t=c$ if and only if 
\begin{equation}\label{eq:criterion}
\gamma'(c)=\vect{0}
\quad
\text{and}
\quad
\det(\gamma''(c),\gamma'''(c))\ne 0
\end{equation}
hold (cf.\ \cite[Theorem 1.3.2]{SUY-book}).
It is known that the curvature function 
$\kappa$ diverges at cusp
\cite[Corollary 1.3.11]{SUY-book}.

Let $\gamma : I \to \R^2$ be a frontal
equipped with a unit tangent vector field $\vt$,
On the regular set $\Reg(\gamma)$,
we set 
$
T={\gamma'}/{\|\gamma'\|}.
%  T=\frac{\gamma'}{\|\gamma'\|},
%  \qquad N = JT,
$
%where $J$ is the $90^\circ$ rotation of $\R^2$.
Then, we define a continuous function 
$\epsilon : \Reg(\gamma) \to \{1,-1\}$ 
as
$$\epsilon(t) = T(t) \cdot \vect{e}(t).$$
We call $\epsilon(t)$ the {\it sign function}.

\begin{lemma}\label{lem:sgn-change}
The sign function $\epsilon(t)$ 
changes its sign across cusp singularities.
\end{lemma}

\begin{proof}
Let $\gamma : I \to \R^2$ be a frontal
having a cusp at $c\in I$.
Without loss of generality, 
we may assume that $c=0$.
The criterion for the cusp \eqref{eq:criterion}
yields that 
$\gamma'(0)=\vect{0}$ and $\gamma''(0)\ne\vect{0}$
hold.
Therefore, there exists a smooth map 
$\vect{x} : I\to \R^2$ such that 
$\gamma'(t) =t\, \vect{x}(t)$.
Since $\gamma''(0)\ne\vect{0}$, 
we have $\vect{x}\ne \vect{0}$ 
on a neighborhood of $t=0$.
We set $E=\vect{x}/\|\vect{x}\|$.
Then $T = \gamma'/\|\gamma'\|$ is written as
$
T(t) = \sgn (t) \, E(t).
$
On the other hand, either
$$
  \text{(i)}~
  \vt(t)= E(t)
  \quad\text{or}\quad  
  \text{(ii)}~
  \vt(t)= -E(t)
$$
holds.
In the case (i), we have $\epsilon(t) = \sgn (t)$.
In the case (ii), we have $\epsilon(t) = -\sgn (t)$.
Therefore, the assertion holds.
\end{proof}

\begin{corollary}\label{cor:sgn-boundary}
Let $\gamma(t)$ be a front in $\R^2$
defined on a bounded closed interval $[a,b]$
$(a<b)$,
and let $j$ be a positive integer.
Suppose that every singular point is cusp,
and the restriction $\gamma|_{(a,b)}$
has exactly $j$ singular points.
Then the sign function $\epsilon(t)$ satisfies
$\epsilon(a)=(-1)^{j} \epsilon(b)$,
where we set 
$\epsilon(a)=\lim_{t\to a+0}\epsilon(t)$
and 
$\epsilon(b)=\lim_{t\to b-0}\epsilon(t)$.
\end{corollary}

\begin{proof}
By Lemma \ref{lem:sgn-change},
we have that $\epsilon(t)$ change its sign at cusps.
Hence, 
$\epsilon(a)=\epsilon(b)$
(resp.\ $\epsilon(a)=-\epsilon(b)$), 
if $j$ is even
(resp.\ odd).
\end{proof}

\begin{lemma}\label{lem:theta-integral}
Let $\gamma(t)$ be a frontal in $\R^2$
defined on a bounded closed interval $[a,b]$
$(a<b)$,
Then it holds that
$
  K(\gamma) \geq \arccos (\vt(a)\cdot\vt(b)).
$
\end{lemma}

\begin{proof}
By a proof similar to Proposition \ref{thm:thm1},
We have $K(\gamma) \geq |\theta(a)-\theta(b)|$.
Since $|\theta(a)-\theta(b)|\geq \arccos (\vt(a)\cdot\vt(b))$,
we have the desired result.
\end{proof}

%%%%%%%%%%%%%%%%%%%%%%%%%%%%%%
%%%%%%%%%%%%%%%%%%%%%%%%%%%%%%
%%%%%%%%%%%%%%%%%%%%%%%%%%%%%%
%%%%%%%%%%%%%%%%%%%%%%%%%%%%%%
\subsection{Angle at the endpoints of frontals}
\label{sec:angle}

The following fact is known (see \cite[Proposition 3.21]{ASSN}, for example), 
cf.\ Figure \ref{fig:cusp1}.

\begin{fact}
\label{fact:sing-0}
Let $\gamma(t)$ be a frontal in $\R^2$
defined on a bounded closed interval $[a,b]$
$(a<b)$. 
If $\gamma(a)=\gamma(b)$,
and the restriction 
$\gamma|_{(a,b)}$
is a regular curve,
then $K(\gamma)> \pi$ holds.
\end{fact}

\begin{figure}[htbp]
\centering
  \resizebox{3.5cm}{!}{\includegraphics{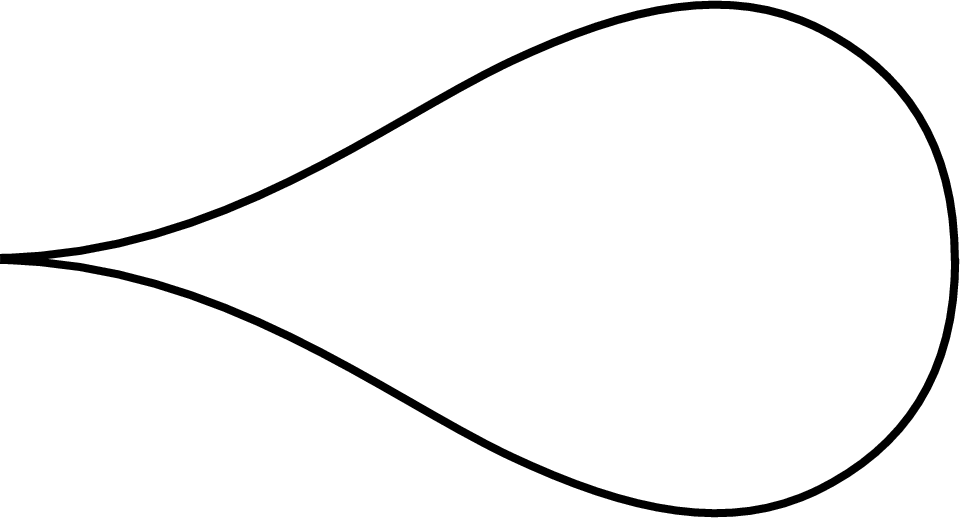}} 
 \caption{
 A frontal having the same endpoints but with no singular point in the interior.
 }
 \label{fig:cusp1}
\end{figure} 

We generalize Fact \ref{fact:sing-0} to frontals
(Propositions \ref{prop:cusp-2} and \ref{prop:cusp-1}).

\begin{definition}
Let $\gamma(t)$ be a frontal in $\R^2$
defined on a bounded closed interval $[a,b]$
$(a<b)$.
Suppose that $\gamma(a)=\gamma(b)$,
and the restriction $\gamma|_{(a,b)}$
has finite singular points.
Then the number $\varphi\in [0,\pi]$ defined by
$$
 \varphi=\arccos\left(-T(a)\cdot T(b) \right)
$$
is called the {\it angle at the endpoints} of $\gamma$,
where $T=\gamma'/\|\gamma'\|$ and
$$
 T(a)=\lim_{t\to a+0} T(t),\qquad
 T(b)=\lim_{t\to b-0} T(t).
$$
\end{definition}

\begin{proposition}\label{prop:cusp-2}
Let $\gamma(t)$ be a frontal in $\R^2$
defined on a bounded closed interval $[a,b]$
$(a<b)$.
Suppose that every singular point is cusp,
$\gamma(a)=\gamma(b)$,
and the restriction $\gamma|_{(a,b)}$
has exactly two singular points.
Then $K(\gamma)\ge \pi-\varphi$ holds,
where $\varphi \in [0,\pi]$ is the angle at endpoints of $\ga$.
\end{proposition}

\begin{proof}
By Corollary \ref{cor:sgn-boundary},
we have $\epsilon(a) = \epsilon(b)$.
Since $T(t)=\epsilon(t)\vt(t)$,
it holds that
$
  T(a)\cdot T(b)
  =\left(\epsilon(a) \vt(a) \right) \cdot \left(\epsilon(b) \vt(b) \right)  
  =\vt(a)\cdot \vt(b).
$
Hence we have
$$
  \varphi 
  = \arccos(-T(a)\cdot T(b))
  = \arccos(-\vt(a)\cdot \vt(b))
  = \pi - \arccos(\vt(a)\cdot \vt(b)).
$$
Since the length $\mathcal{L}(\vt)$ of $\vt$ 
satisfies $\mathcal{L}(\vt)\geq\arccos(\vt(a)\cdot \vt(b))$,
Lemma \ref{lem:t-length} yields that
$K(\gamma)=\mathcal{L}(\vect{e})\geq \pi-\varphi$.
\end{proof}

\begin{figure}[htbp]
\centering
  \resizebox{3.5cm}{!}{\includegraphics{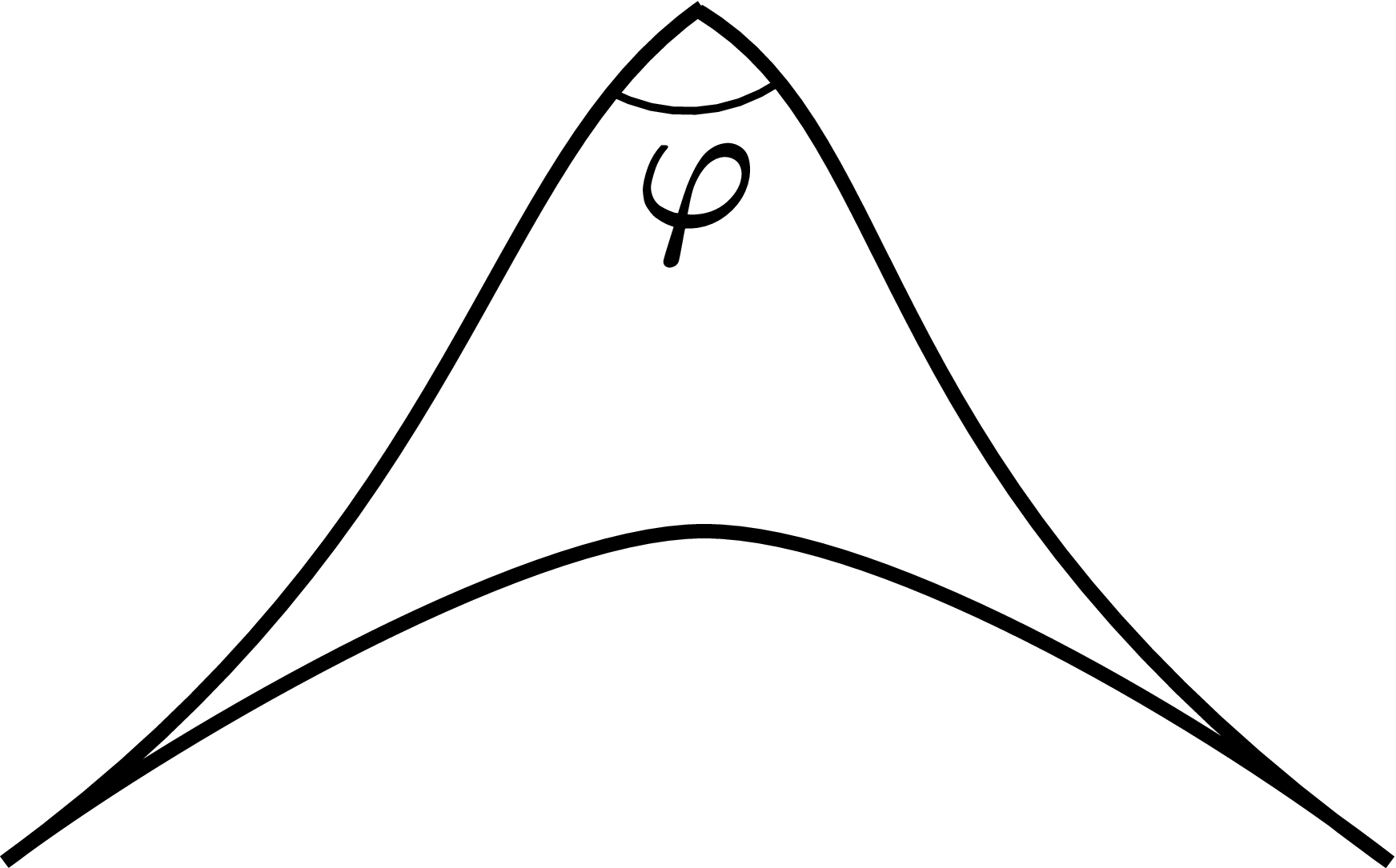}} 
 \caption{
 A frontal having the same endpoints but with two singular points in the interior.
 }
 \label{fig:2-sing}
\end{figure} 

\begin{proposition}\label{prop:cusp-1}
Let $\gamma(t)$ be a frontal in $\R^2$
defined on a bounded closed interval $[a,b]$
$(a<b)$.
Suppose that every singular point is cusp,
$\gamma(a)=\gamma(b)$,
and the restriction $\gamma|_{(a,b)}$
has exactly one singular point.
Then $K(\gamma)>\varphi$ holds,
where $\varphi \in [0,\pi]$ is the angle at endpoints of $\ga$.
\end{proposition}

\begin{proof}
Let $c\in (a,b)$ be the singular point of $\gamma(t)$.
By Corollary \ref{cor:sgn-boundary},
we have $\epsilon(a) = -\epsilon(b)$.
Since $T(t)=\epsilon(t)\vt(t)$,
it holds that
$$
  \varphi = \arccos(-T(a)\cdot T(b))
  = \arccos(\vt(a)\cdot\vt(b)).
$$

If necessary, by a rotation and translation of $\R^2$,
we may assume that
$\gamma(c)=(0,0)$ and 
$\vt(c)=(0,1)$ hold.
We denote by $\gamma(t)=(x(t),y(t))$.
Since $x(t)$ is a continuous function on a closed interval $[a,b]$,
there exist $t_1, t_2 \in [a,b]$ such that
$\max_{t\in [a,b]} x(t) = x(t_1)$ and $\min_{t\in [a,b]} x(t) = x(t_2)$ hold.
If $t_1 =t_2$,
then $\max_{t\in [a,b]} x(t)=\min_{t\in [a,b]} x(t)$ holds.
Hence we have $x(t)$ is identically zero, 
and the image of $\gamma(t)$
is a subset of the $y$-axis.
Then the curvature function is identically zero,
which contradicts the fact that the curvature diverges at cusps
(\cite[Corollary 1.3.11]{SUY-book}).
Hence $t_1\ne t_2$ holds.
If necessary,
by changing the orientation of $t$, 
we may assume that $t_1<t_2$. 

Also we have that $t_1,t_2\ne c$.
In fact, if $t_1=c$, then $x(t)\leq 0$ holds.
The tangent line\footnote{The tangent line of a front at a cusp
is also called the {\it centerline} \cite[Corollary 1.3.11]{SUY-book}.}
of $\gamma(t)$ at the cusp $t=c$
divides the curve $\gamma(t)$ into two parts, 
one on each side
(\cite[Corollary 1.3.11]{SUY-book}).
Since the centerline of $\gamma(t)$ at $t=c$ is the $y$-axis, 
\cite[Corollary 1.3.11]{SUY-book} yields a contradiction,
and hence $t_1\ne c$.
Similarly, we have $t_2\ne c$.

We set
$$
\arccos (\vt(a)\cdot \vt(c)) = \varphi_a, \qquad
\arccos (\vt(b)\cdot \vt(c)) = \varphi_b.
$$
Then, $\varphi_a+\varphi_b\geq\varphi$ holds.

\medskip
\noindent
\underline{\bf (I)~The case of $a<t_1<c<t_2<b$.}\quad
The total absolute curvature $K(\gamma)$ is divided into
\begin{equation}\label{eq:case-I}
K(\gamma)=
K(\gamma|_{[a,t_1]})+
K(\gamma|_{[t_1,c]})+
K(\gamma|_{[c,t_2]})+
K(\gamma|_{[t_2,b]}).
\end{equation}
Then it holds that
\begin{equation}\label{eq:case-I-ep}
K(\gamma|_{[t_1,c]})>0\quad \text{and}\quad K(\gamma|_{[c,t_2]})>0.
\end{equation}
In fact, if we suppose $K(\gamma|_{[t_1,c]})=0$,
then $\kappa(t)$ is identically zero on $[t_1,c]$,
which contradicts the fact that the curvature diverges at cusps
(\cite[Corollary 1.3.11]{SUY-book}).
Hence, we have $K(\gamma|_{[t_1,c]})>0$.
Similarly, $K(\gamma|_{[c,t_2]})>0$ holds.

Since $x'(t_1)=x'(t_2)=0$,
we have $\vt =(0,\pm1)$ at $t_1, t_2$,
namely, $\vt(t_i) =\pm \vt(c)$ holds $(i=1,2)$.
Thus, either (I-i), (I-ii), (I-iii) or (I-iv) occurs,
cf.\ Figure \ref{fig:caseI}.
\begin{itemize}
\item[(I-i)] If $\vt(t_1) =\vt(t_2) =\vt(c)$,
$K(\gamma|_{[a,t_1]})\geq \varphi_a$ and
$K(\gamma|_{[t_2,b]})\geq \varphi_b$ hold.
By \eqref{eq:case-I} and \eqref{eq:case-I-ep},
we have
$
K(\gamma)>
K(\gamma|_{[a,t_1]})+
K(\gamma|_{[t_2,b]})
\geq
\varphi_a + \varphi_b
\geq \varphi.
$

\item[(I-ii)] If $\vt(t_1) =-\vt(t_2) =\vt(c)$,
Lemma \ref{lem:theta-integral} yields
$K(\gamma|_{[c,t_2]})\geq \arccos (\vt(c)\cdot \vt(t_2))=\pi$. 
By \eqref{eq:case-I} and \eqref{eq:case-I-ep},
it holds that
$K(\gamma)>\pi \geq \varphi.$
\item[(I-iii)] If $-\vt(t_1) =\vt(t_2) =\vt(c)$,
as in the case (I-ii), we have
$K(\gamma|_{[t_1,c]})\geq \pi$.
By \eqref{eq:case-I} and \eqref{eq:case-I-ep},
it holds that
$K(\gamma)>\pi \geq \varphi.$
\item[(I-iv)] If $\vt(t_1) =\vt(t_2) =-\vt(c)$,
as in the cases (I-ii) and (I-iii),
we have
$K(\gamma|_{[t_1,c]})\geq \pi$.
By \eqref{eq:case-I} and \eqref{eq:case-I-ep},
it holds that
$K(\gamma)>\pi \geq \varphi.$
\end{itemize}
Therefore, in all cases (I-i), (I-ii), (I-iii), and (I-iv),
we have
$K(\gamma)>\varphi$.

\begin{figure}[htbp]
 \begin{tabular}{@{\hspace{-4mm}}c@{\hspace{4mm}}c@{\hspace{5mm}}c@{\hspace{4mm}}c}
 %{ccc}
  \mbox{\raisebox{3.5mm}{\includegraphics[width=30mm]{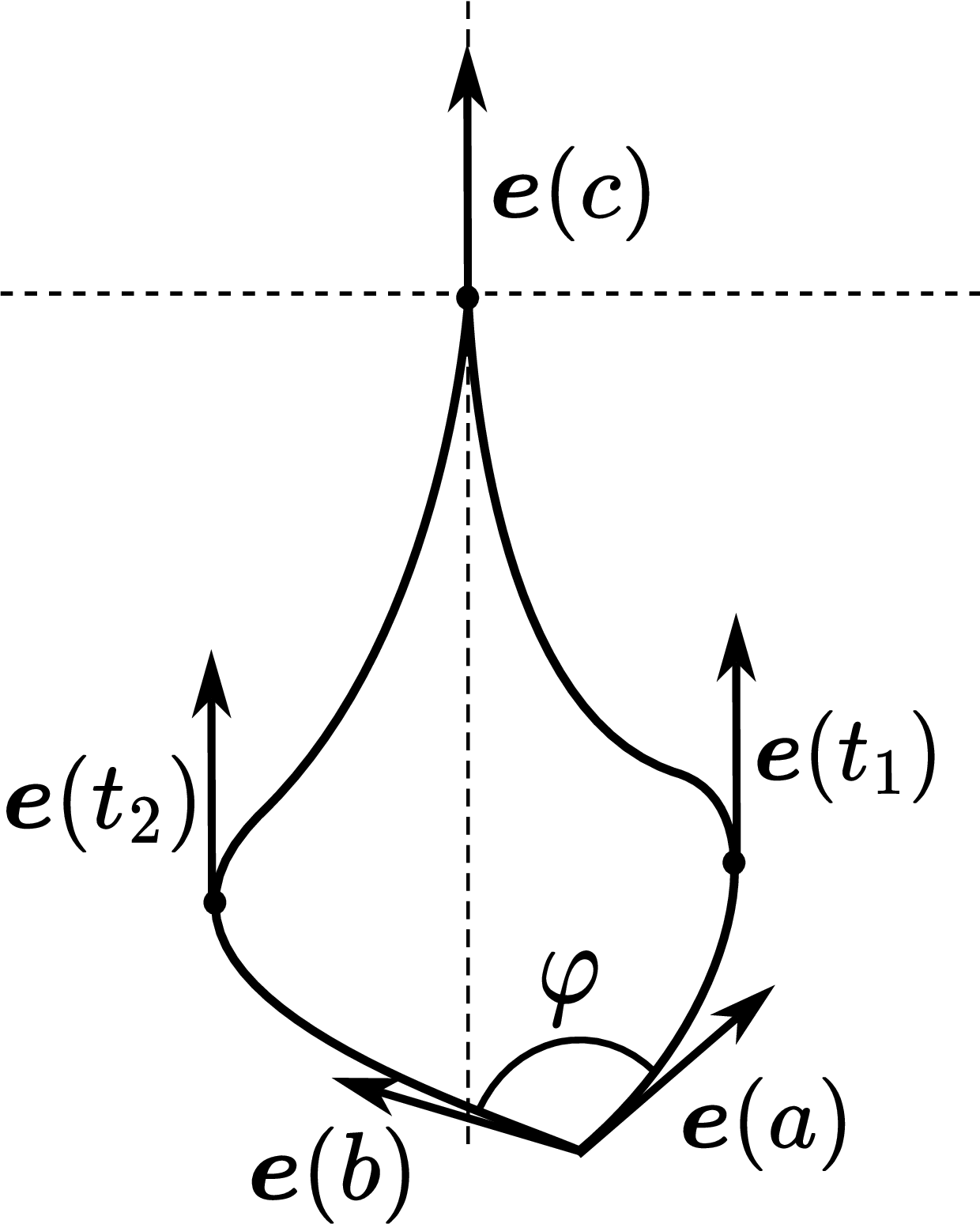}}} &
  \resizebox{3cm}{!}{\includegraphics{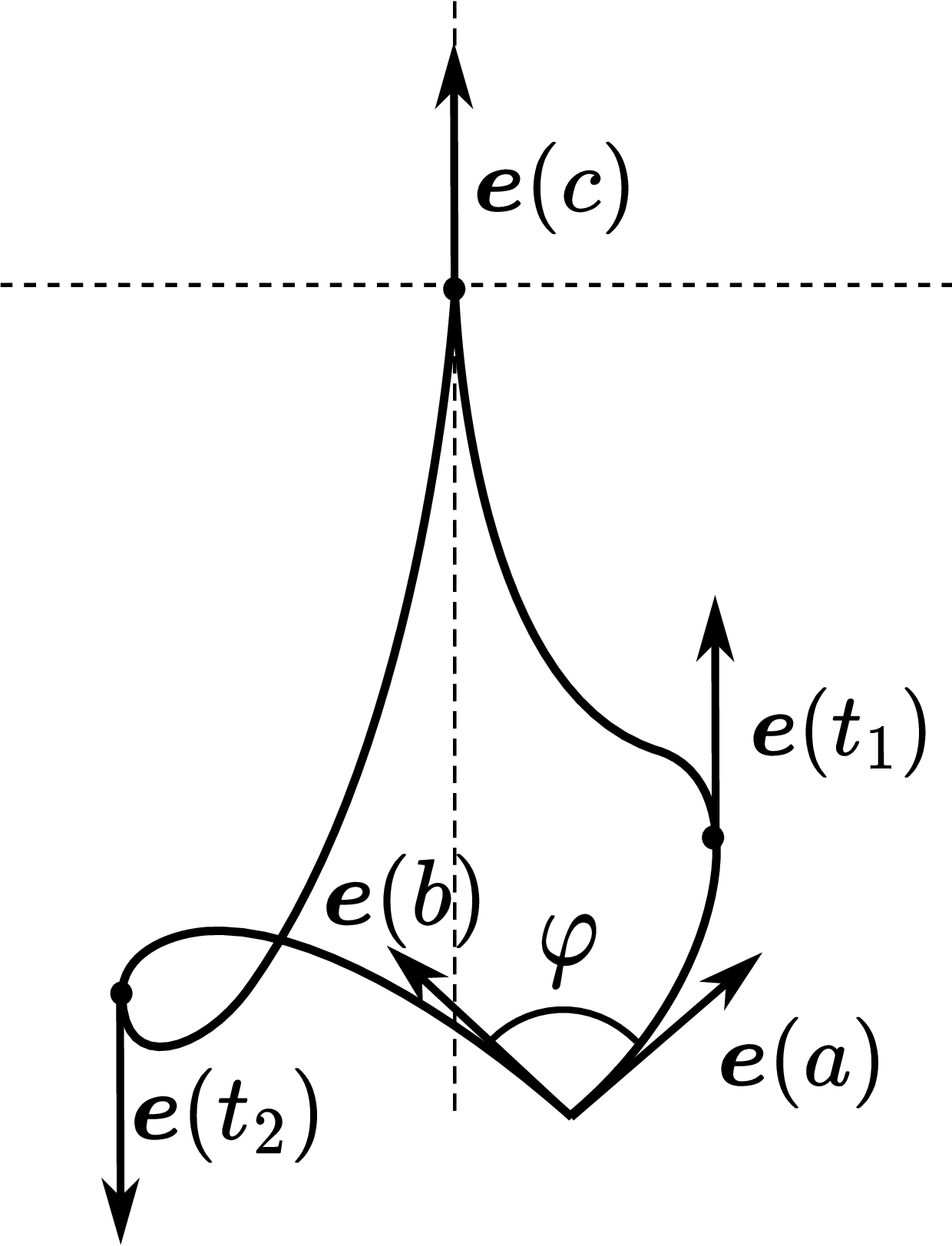}} &
  \mbox{\raisebox{4.4mm}{\includegraphics[width=32mm]{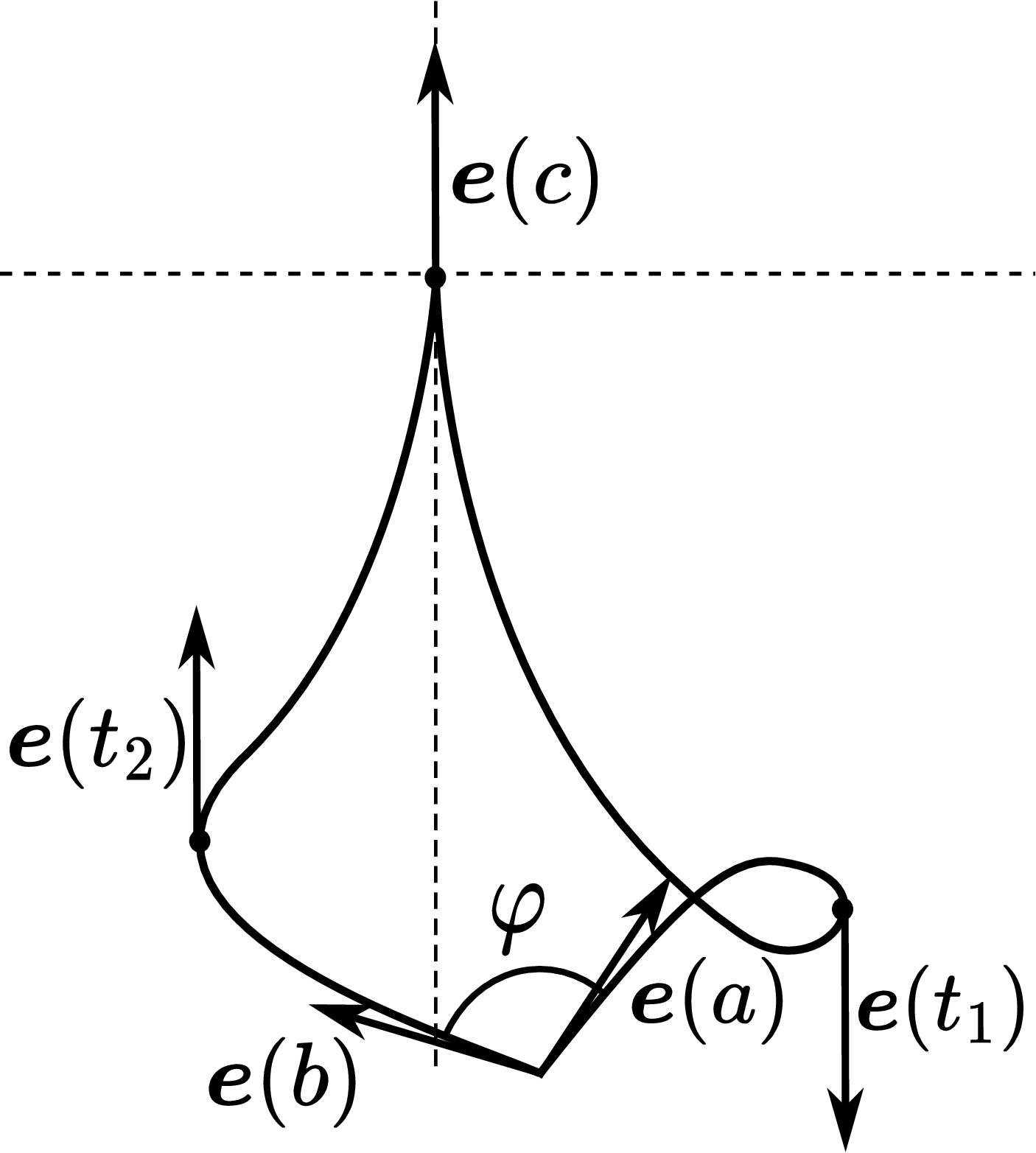}}} &
  \mbox{\raisebox{8mm}{\includegraphics[width=32mm]{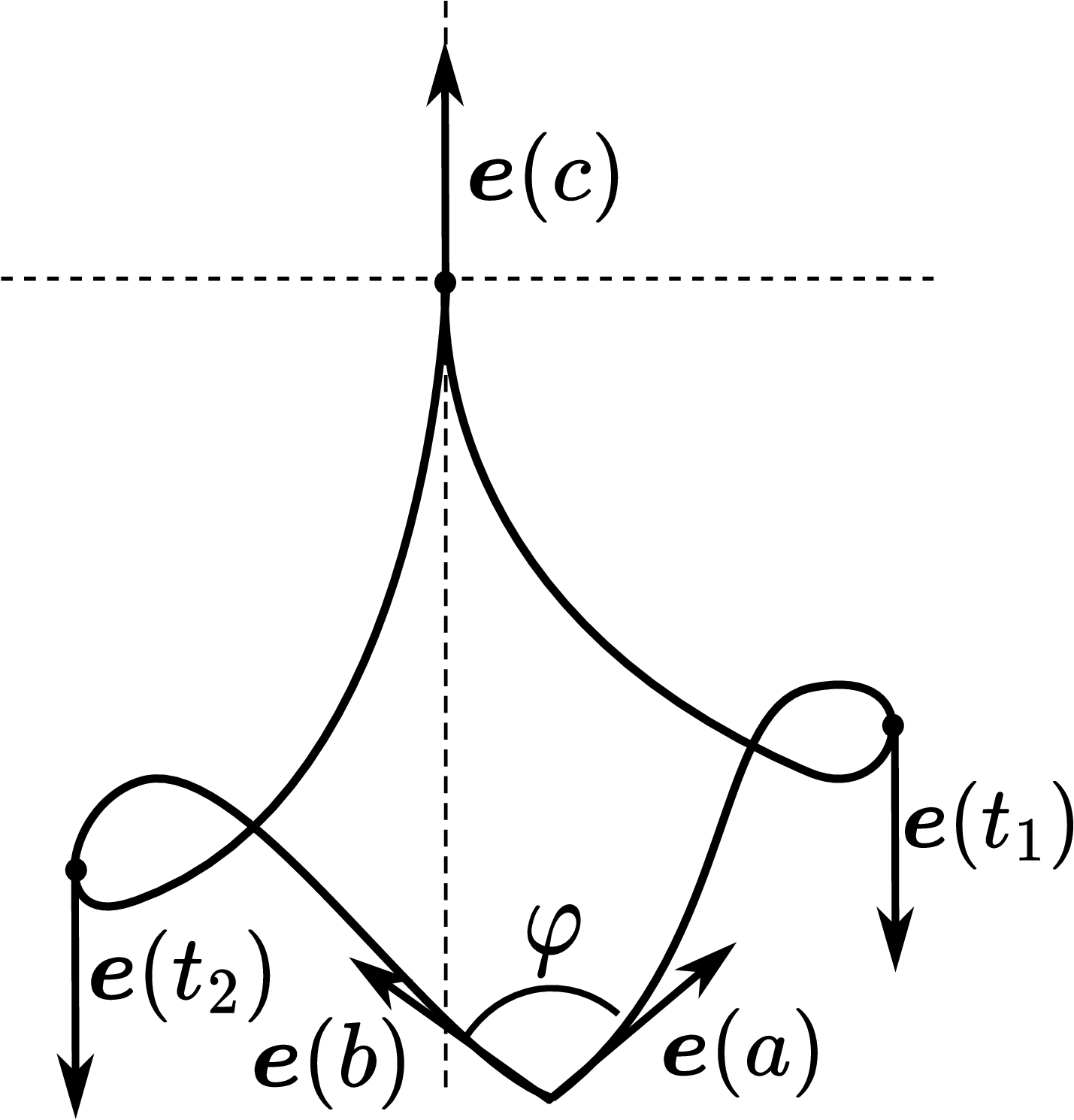}}} 
  \vspace{6mm}\\
  {\footnotesize Case (I-i)} &
  {\footnotesize Case (I-ii)} &
  {\footnotesize Case (I-iii)} &
  {\footnotesize Case (I-iv)} 
 \end{tabular}  
 \caption{
 Frontals having the same endpoints but with one singular point in the interior,
 the case (I).
 }
 \label{fig:caseI}
\end{figure}

\medskip
\noindent
\underline{\bf (II)~The case of $a<t_1<t_2<c<b$.}\quad
The total absolute curvature 
$K(\gamma)$ is written as
$
K(\gamma)=
K(\gamma|_{[a,t_1]})+
K(\gamma|_{[t_1,t_2]})+
K(\gamma|_{[t_2,c]})+
K(\gamma|_{[c,b]}).
$
As in \eqref{eq:case-I-ep}, 
we have
\begin{equation}\label{eq:case-II-ep}
K(\gamma|_{[t_2,c]})>0.
\end{equation}
Similarly,  
\begin{equation}\label{eq:case-II-ep2}
K(\gamma|_{[t_1,t_2]})>0
\end{equation}
holds. 
Then either (II-i), (II-ii), (II-iii) or (II-iv) occurs,
cf.\ Figure \ref{fig:caseII}.

\begin{itemize}
\item[(II-i)] If $\vt(t_1) =\vt(t_2) =\vt(c)$,
$K(\gamma|_{[a,t_1]})\geq \varphi_a$ and
$K(\gamma|_{[c,b]})\geq \varphi_b$ hold.
By \eqref{eq:case-II-ep},
we have
$
K(\gamma)>
K(\gamma|_{[a,t_1]})+
K(\gamma|_{[c,b]})
\geq
\varphi_a + \varphi_b
\geq \varphi.
$
\item[(II-ii)] If $\vt(t_1) =-\vt(t_2) =\vt(c)$,
Lemma \ref{lem:theta-integral} yields
$K(\gamma|_{[t_1,t_2]})\geq \arccos (\vt(t_1)\cdot \vt(t_2))=\pi$
and
$K(\gamma|_{[c,t_2]})\geq \arccos (\vt(c)\cdot \vt(t_2))=\pi.$
Hence,
it holds that
$K(\gamma)\geq2\pi> \varphi.$
\item[(II-iii)] If $-\vt(t_1) =\vt(t_2) =\vt(c)$,
as in the case (II-ii), we have
$K(\gamma|_{[t_1,t_2]})\geq \pi$.
By \eqref{eq:case-II-ep},
it holds that
$K(\gamma)>\pi \geq \varphi.$
\item[(II-iv)] If $\vt(t_1) =\vt(t_2) =-\vt(c)$,
as in the cases (II-ii) and (II-iii),
we have
$K(\gamma|_{[t_2,c]})\geq \pi$.
By \eqref{eq:case-II-ep2},
it holds that
$K(\gamma)>\pi \geq \varphi.$
\end{itemize}
Therefore, in all cases (II-i), (II-ii), (II-iii), and (II-iv),
we have
$K(\gamma)>\varphi$.

\begin{figure}[htbp]
 \begin{tabular}{@{\hspace{-4mm}}c@{\hspace{4mm}}c@{\hspace{5mm}}c@{\hspace{4mm}}c}
 %{ccc}
  \mbox{\raisebox{1.4mm}{\includegraphics[width=30mm]{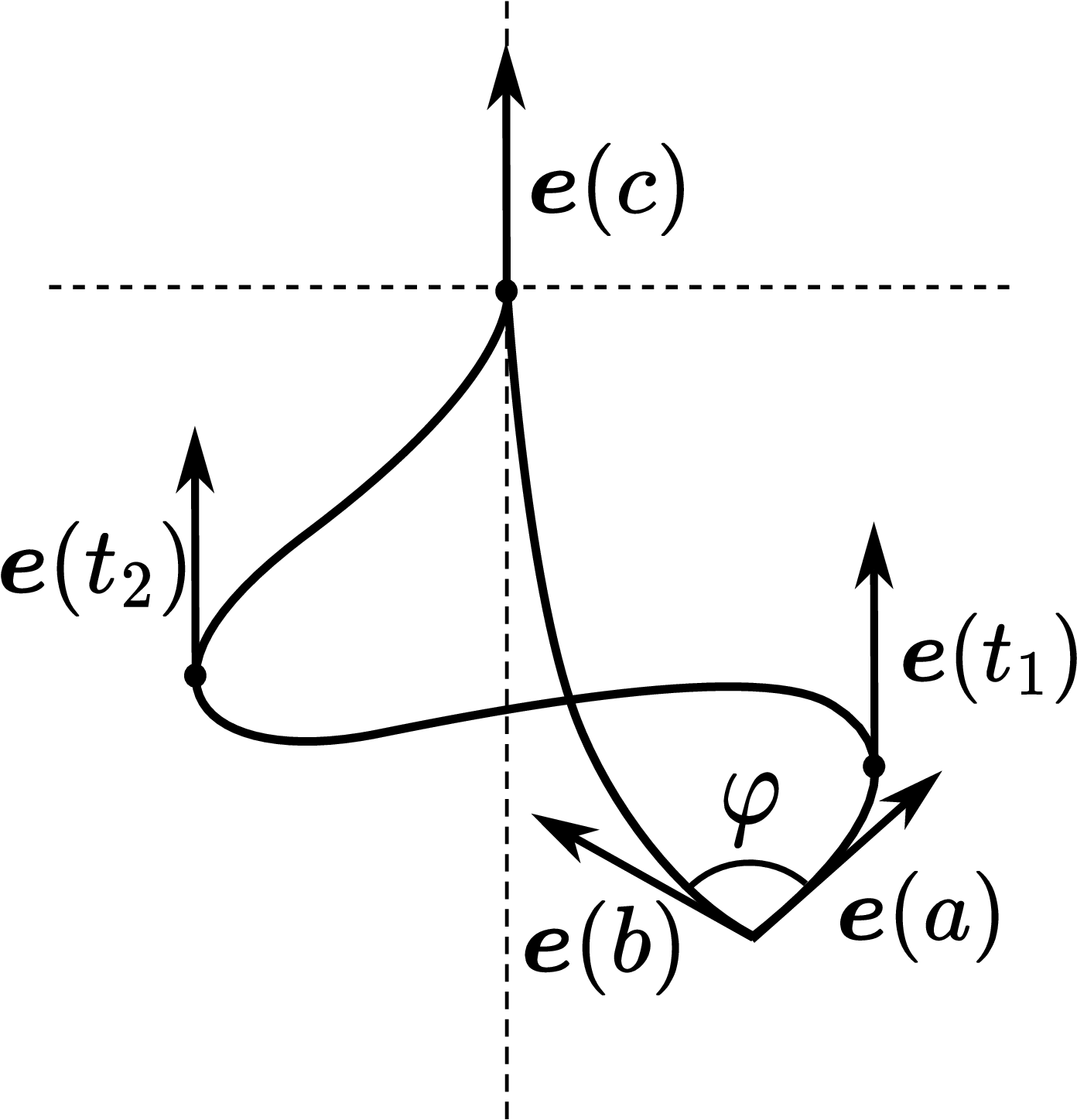}}} &
  \mbox{\raisebox{0.8mm}{\includegraphics[width=30mm]{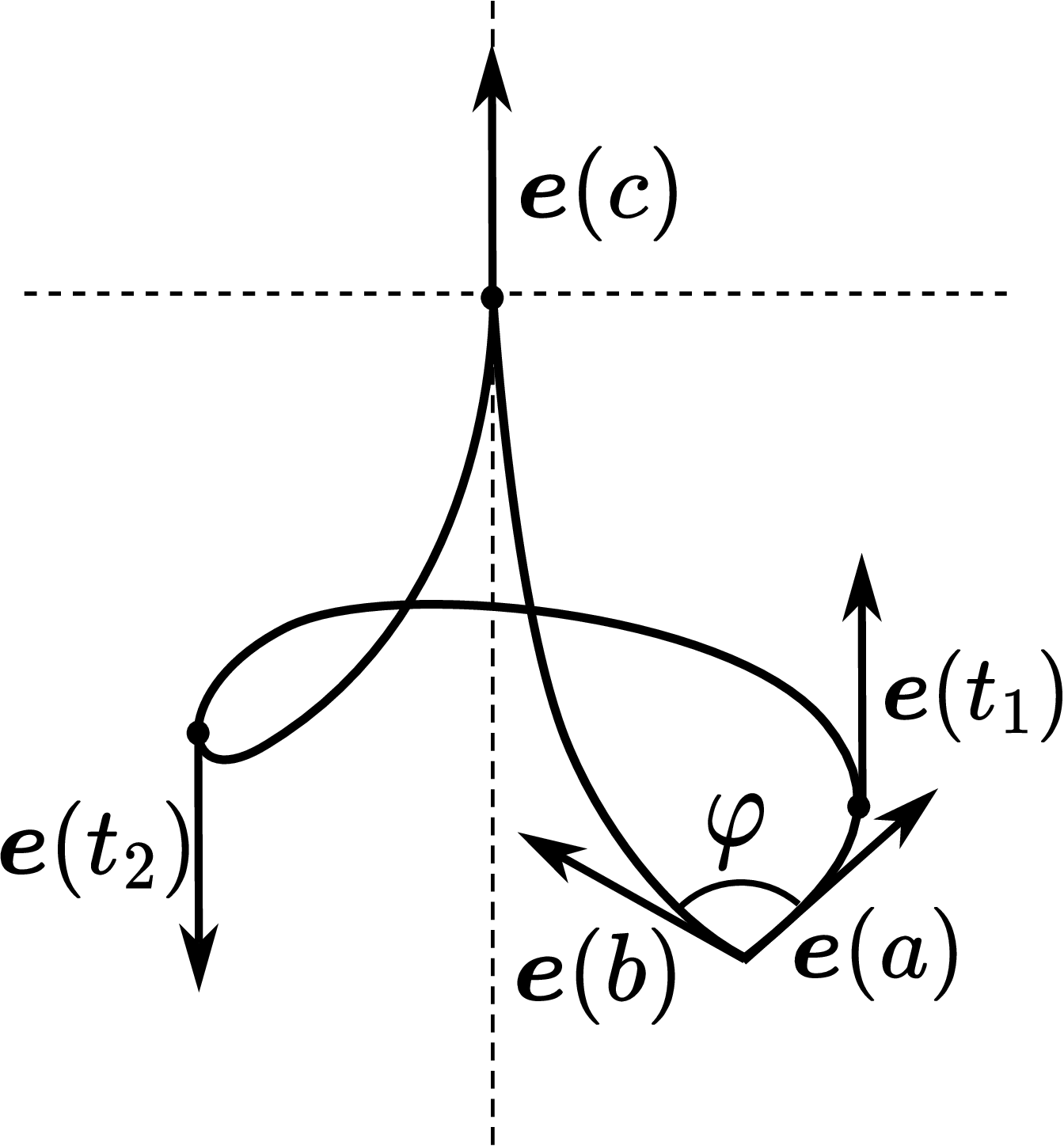}}} &
  \mbox{\raisebox{0mm}{\includegraphics[width=32mm]{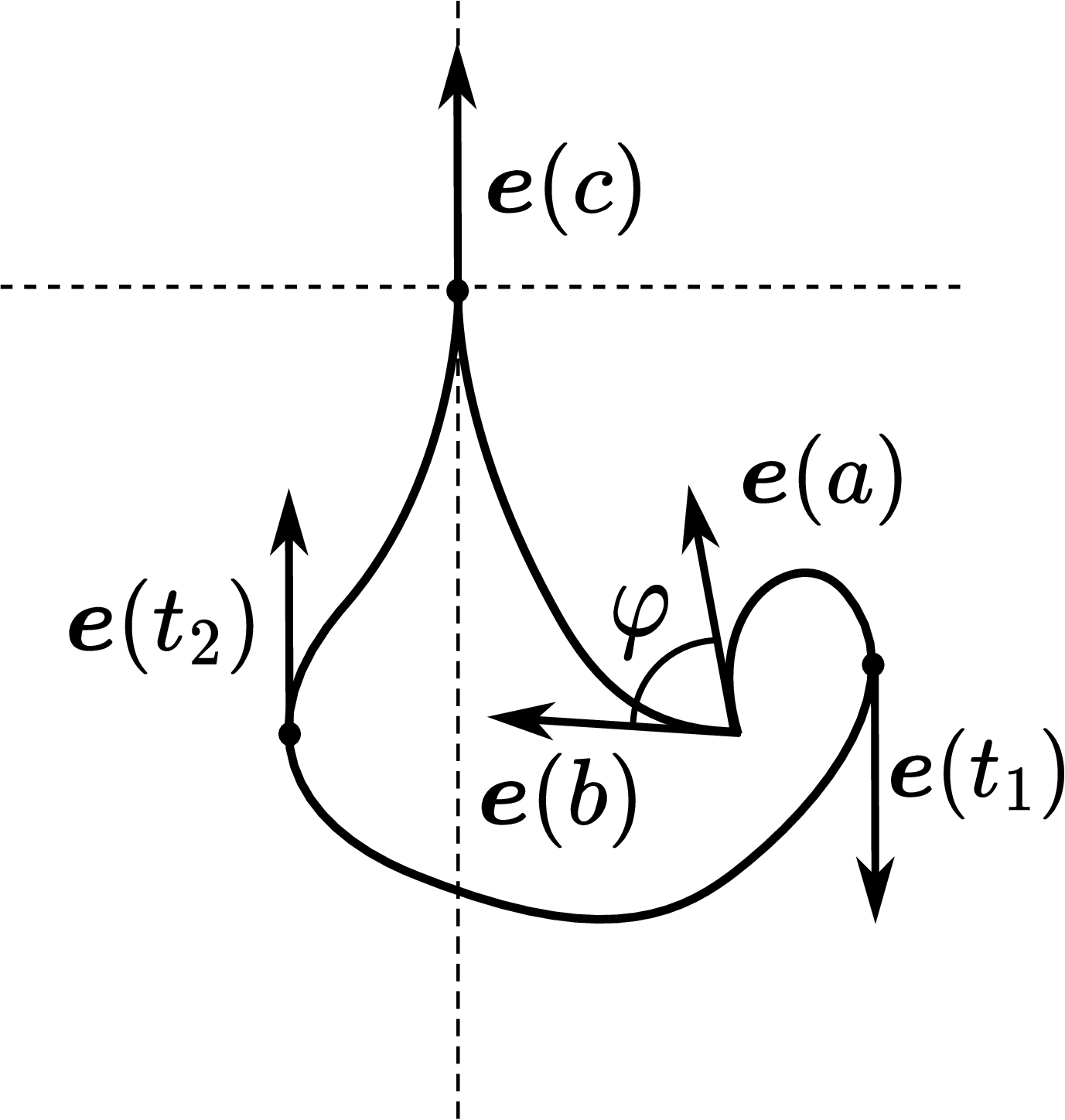}}} &
  \mbox{\raisebox{0.4mm}{\includegraphics[width=32mm]{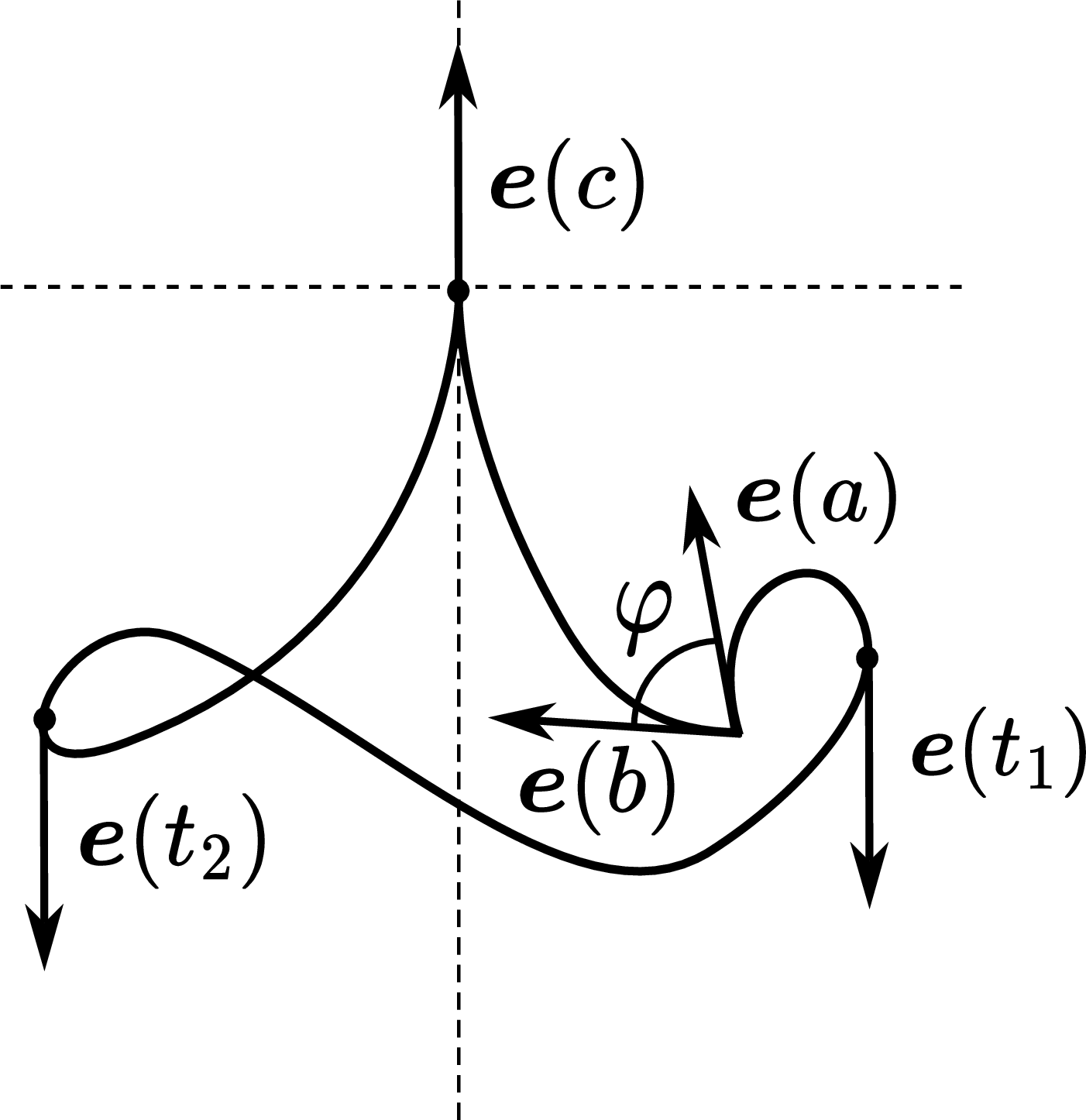}}} 
  \vspace{6mm}\\
  {\footnotesize Case (II-i)} &
  {\footnotesize Case (II-ii)} &
  {\footnotesize Case (II-iii)} &
  {\footnotesize Case (II-iv)} 
 \end{tabular}  
 \caption{
 Frontals having the same endpoints but with one singular point in the interior,
 the case (II).
 }
 \label{fig:caseII}
\end{figure}

\medskip
\noindent
\underline{\bf (III)~The case of $a=t_1<c<t_2<b$.}\quad
The total absolute curvature $K(\gamma)$ 
is written as
$
K(\gamma)=
K(\gamma|_{[a,c]})+
K(\gamma|_{[c,t_2]})+
K(\gamma|_{[t_2,b]}).
$
As in \eqref{eq:case-I-ep}, 
we have
$
K(\gamma|_{[c,t_2]})>0
$
and 
$K(\gamma|_{[a,c]})>0$.
As in the cases (I), (II),
since $x'(t_2)=0$,
we have $\vt(t_2) =\pm \vt(c)$.
Thus, either (III-i) or (III-ii) occurs,
cf.\ Figure \ref{fig:caseIII}.
\begin{itemize}
\item[(III-i)] If $\vt(t_2) =\vt(c)$,
$K(\gamma|_{[a,c]})\geq \varphi_a$ and
$K(\gamma|_{[t_2,b]})\geq \varphi_b$ hold.
By $K(\gamma|_{[c,t_2]})>0$,
we have
$
K(\gamma)>
K(\gamma|_{[a,c]})+
K(\gamma|_{[t_2,b]})
\geq
\varphi_a + \varphi_b
\geq \varphi.
$
\item[(III-ii)] If $\vt(t_2) = - \vt(c)$,
we have 
$K(\gamma|_{[c,t_2]})\geq \arccos (\vt(c)\cdot \vt(t_2))=\pi$.
Hence, by $K(\gamma|_{[a,c]})>0$, we have
$K(\gamma) > \pi \geq \varphi$.
\end{itemize}

\begin{figure}[htbp]
 \begin{tabular}{@{\hspace{-4mm}}c@{\hspace{4mm}}c@{\hspace{5mm}}c@{\hspace{4mm}}c}
 %{ccc}
  \mbox{\raisebox{1.4mm}{\includegraphics[width=30mm]{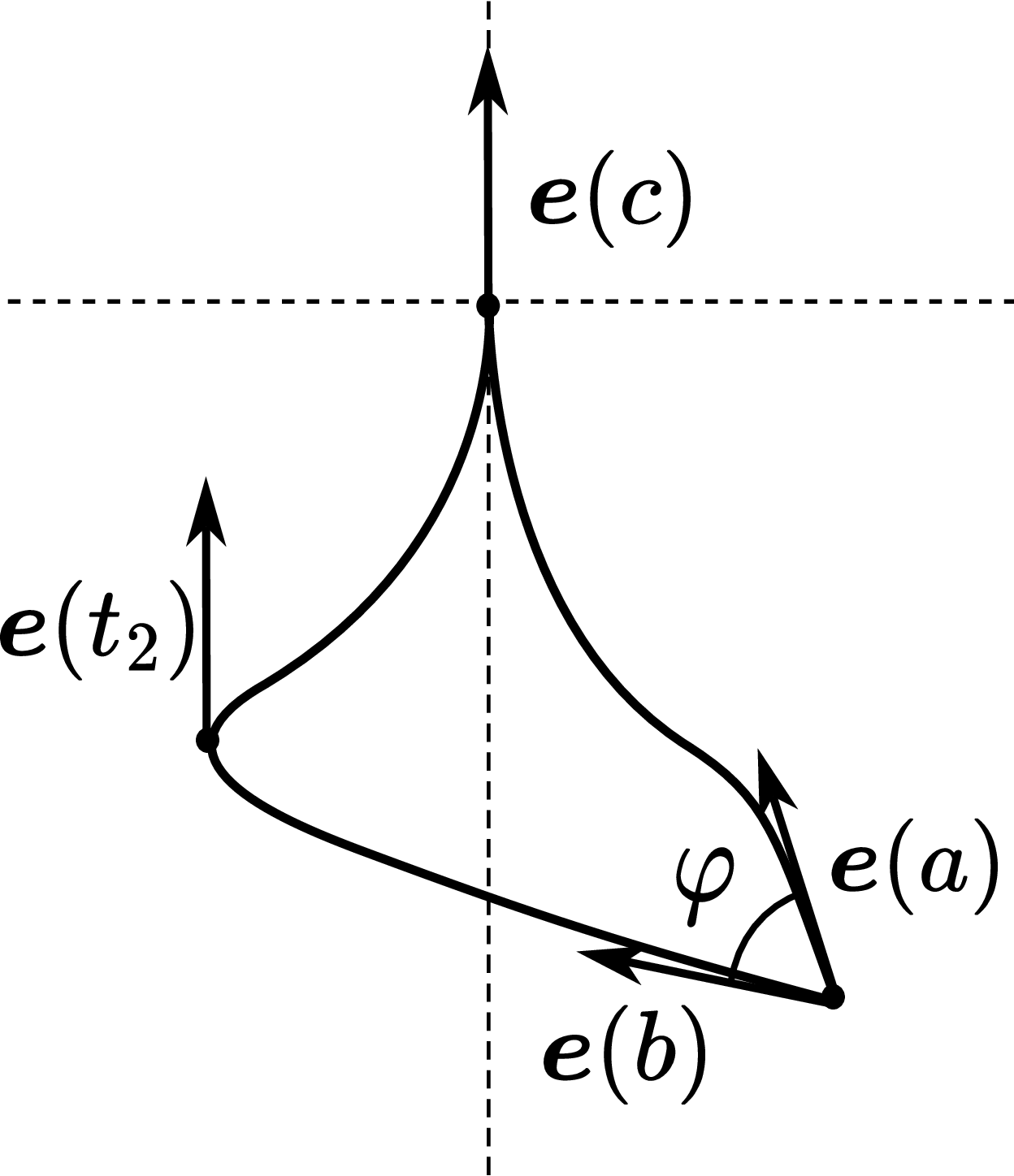}}} &
  \mbox{\raisebox{0.8mm}{\includegraphics[width=30mm]{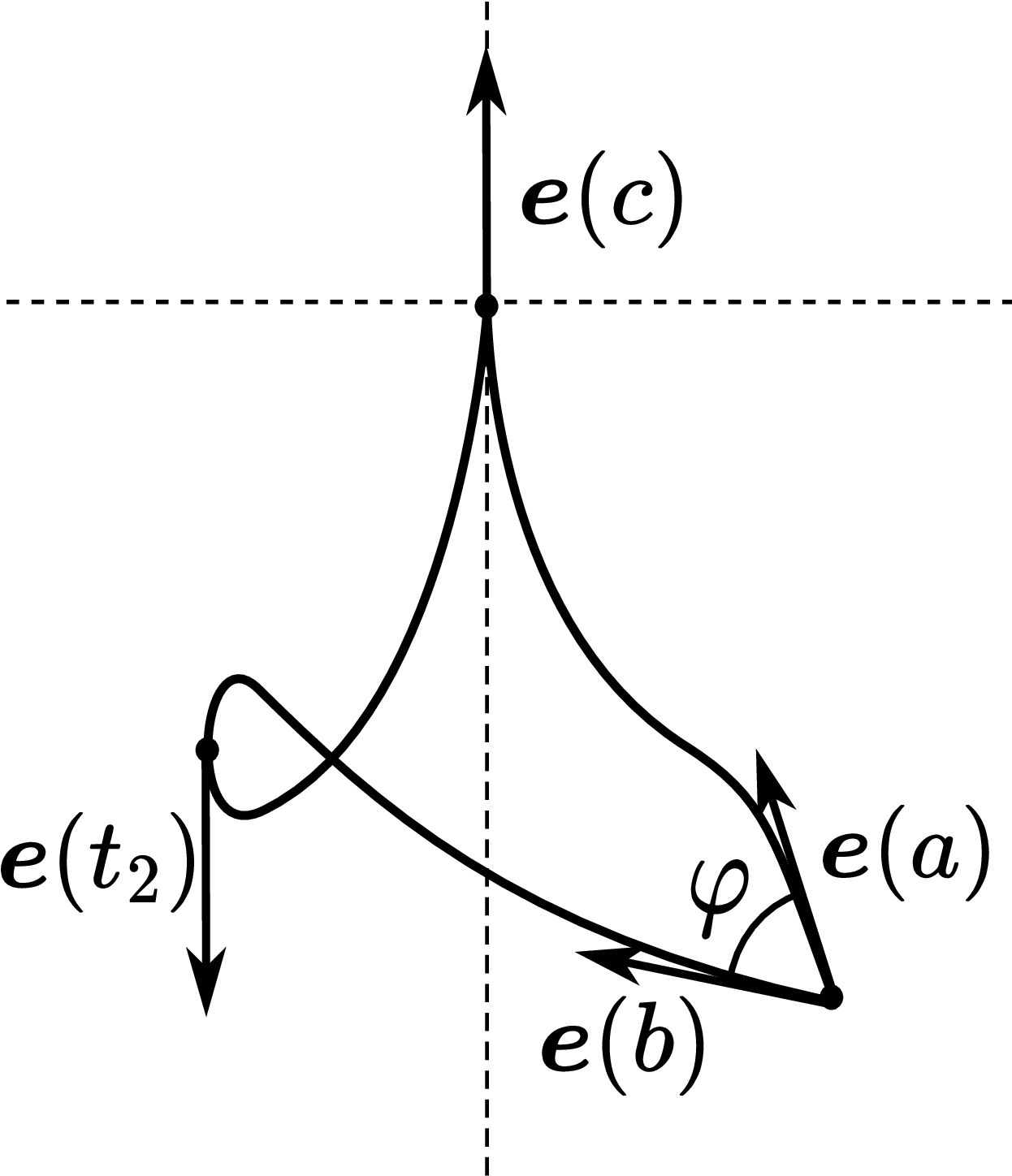}}} &
  \mbox{\raisebox{0mm}{\includegraphics[width=32mm]{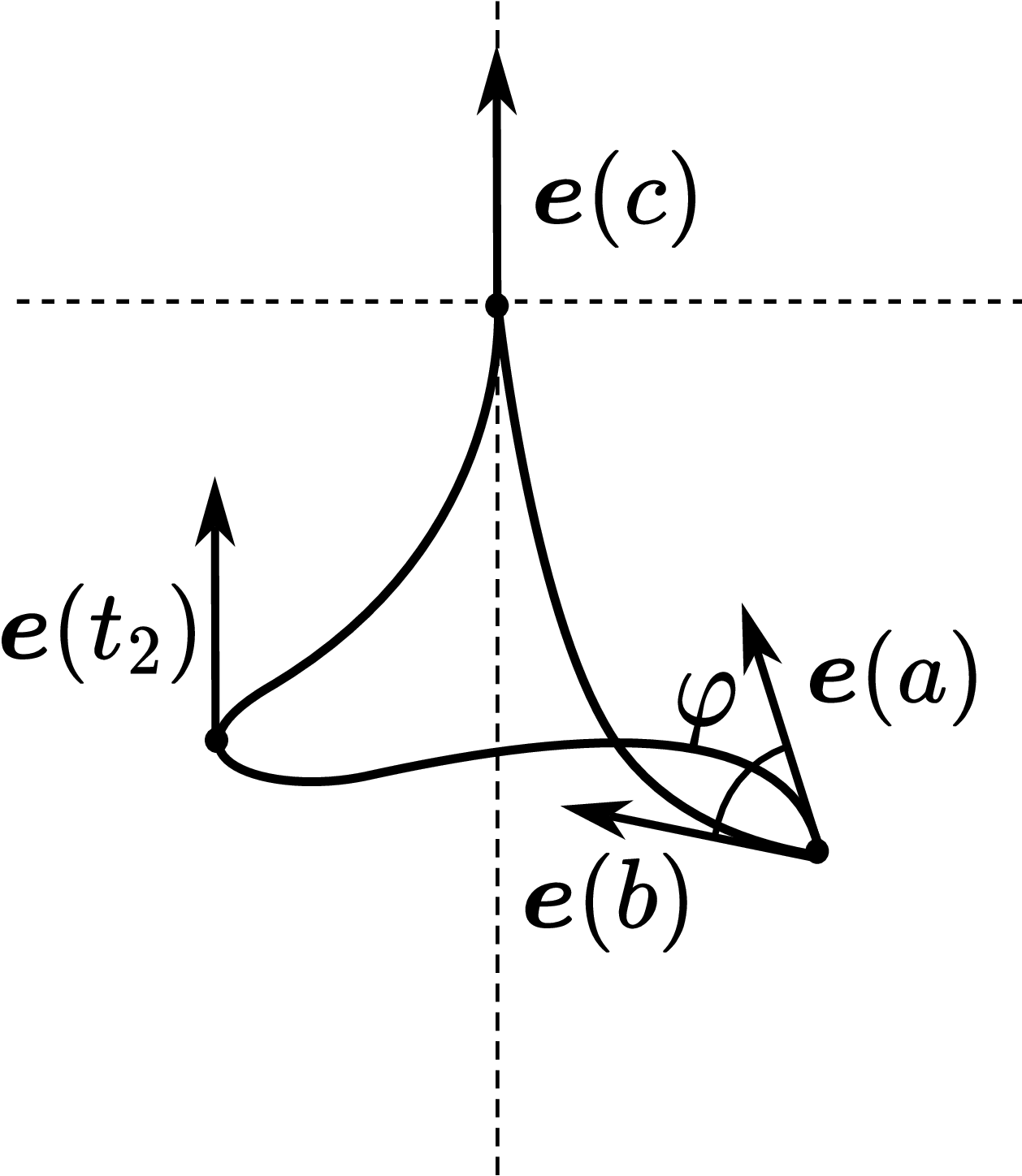}}} &
  \mbox{\raisebox{0.4mm}{\includegraphics[width=32mm]{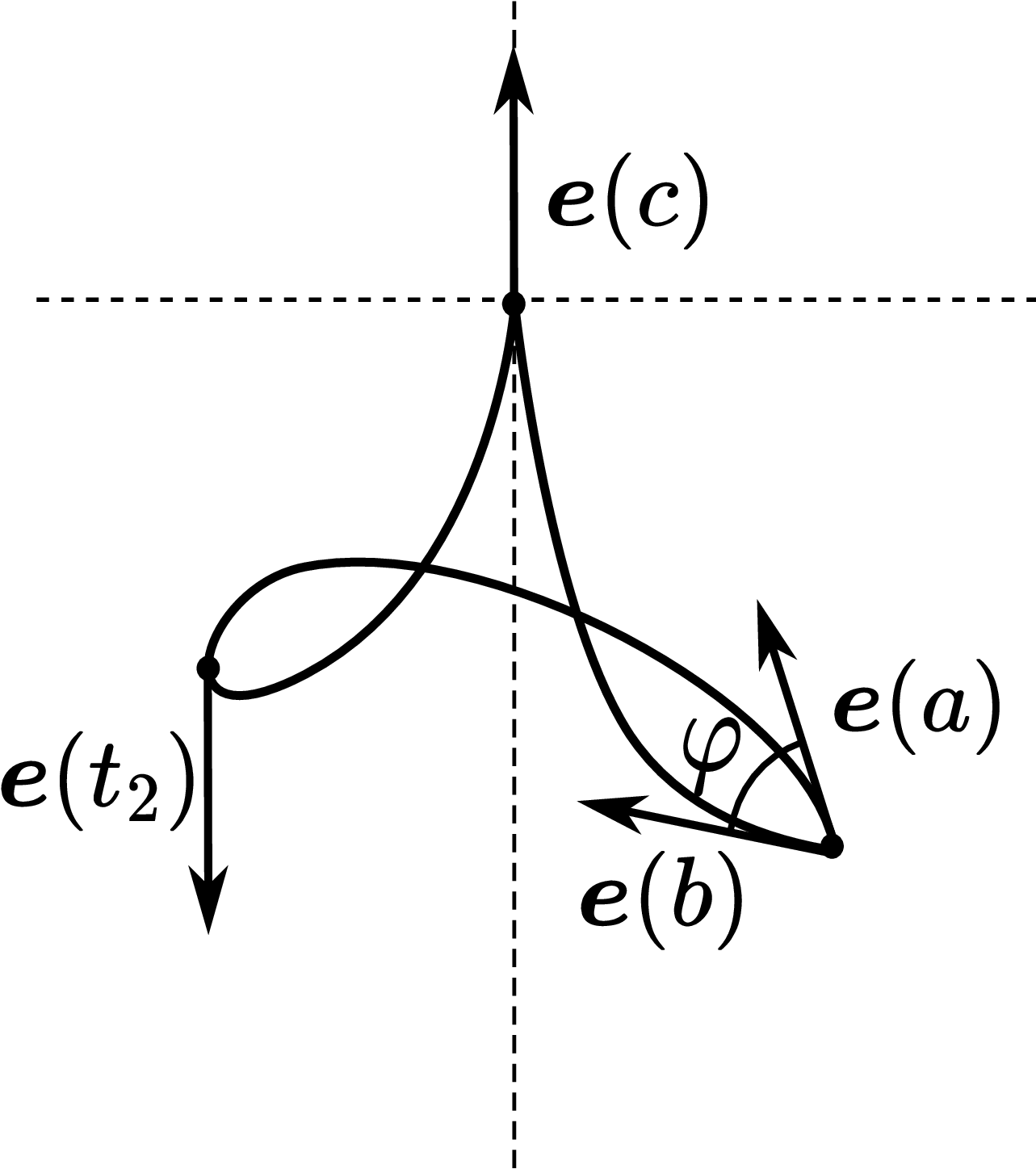}}} 
  \vspace{6mm}\\
  {\footnotesize Case (III-i)} &
  {\footnotesize Case (III-ii)} &
  {\footnotesize Case (IV-i)} &
  {\footnotesize Case (IV-ii)} 
 \end{tabular}  
 \caption{
 Frontals having the same endpoints but with one singular point in the interior,
 the cases (III) and (IV).
 }
 \label{fig:caseIII}
\end{figure} 

\medskip
\noindent
\underline{\bf (IV)~The case of $a=t_1<t_2<c<b$.}\quad
The total absolute curvature $K(\gamma)$ is written as
$
K(\gamma)=
K(\gamma|_{[a,t_2]})+
K(\gamma|_{[t_2,c]})+
K(\gamma|_{[c,b]}).
$
As in \eqref{eq:case-I-ep}, 
we have
$K(\gamma|_{[t_2,c]})>0$
and 
$K(\gamma|_{[c,b]})>0$.

As in the cases (I), (II), (III),
since $x'(t_2)=0$,
we have $\vt(t_2) =\pm \vt(c)$.
Thus, either (IV-i) or (IV-ii) occurs,
cf.\ Figure \ref{fig:caseIII}.
\begin{itemize}
\item[(IV-i)] If $\vt(t_2) =\vt(c)$,
$K(\gamma|_{[a,t_2]})\geq \varphi_a$ and
$K(\gamma|_{[c,b]})\geq \varphi_b$ hold.
By $K(\gamma|_{[t_2,c]})>0$,
we have
$
K(\gamma)>
K(\gamma|_{[a,t_2]})+
K(\gamma|_{[c,b]})
\geq
\varphi_a + \varphi_b
\geq \varphi.
$
\item[(IV-ii)] If $\vt(t_2) = - \vt(c)$,
Lemma \ref{lem:theta-integral} yields
$K(\gamma|_{[t_2,c]})\geq \arccos (\vt(t_2)\cdot \vt(c))=\pi$.
By $K(\gamma|_{[c,b]})>0$,
we have
$
K(\gamma)>
\pi
\geq \varphi.
$
\end{itemize}

\medskip
\noindent
\underline{\bf (V)~The case of $a<t_1<c<t_2=b$.}\quad
By an argument similar to the case (III),
we have $K(\gamma)>\varphi$,
cf.\ Figure \ref{fig:caseV}.

\medskip
\noindent
\underline{\bf (VI)~The case of $a<c<t_1<t_2=b$.}\quad
By an argument similar to the case (IV),
we have $K(\gamma)>\varphi$,
cf.\ Figure \ref{fig:caseV}.

\begin{figure}[htbp]
 \begin{tabular}{@{\hspace{-4mm}}c@{\hspace{4mm}}c@{\hspace{5mm}}c@{\hspace{4mm}}c}
 %{ccc}
  \mbox{\raisebox{1.4mm}{\includegraphics[width=30mm]{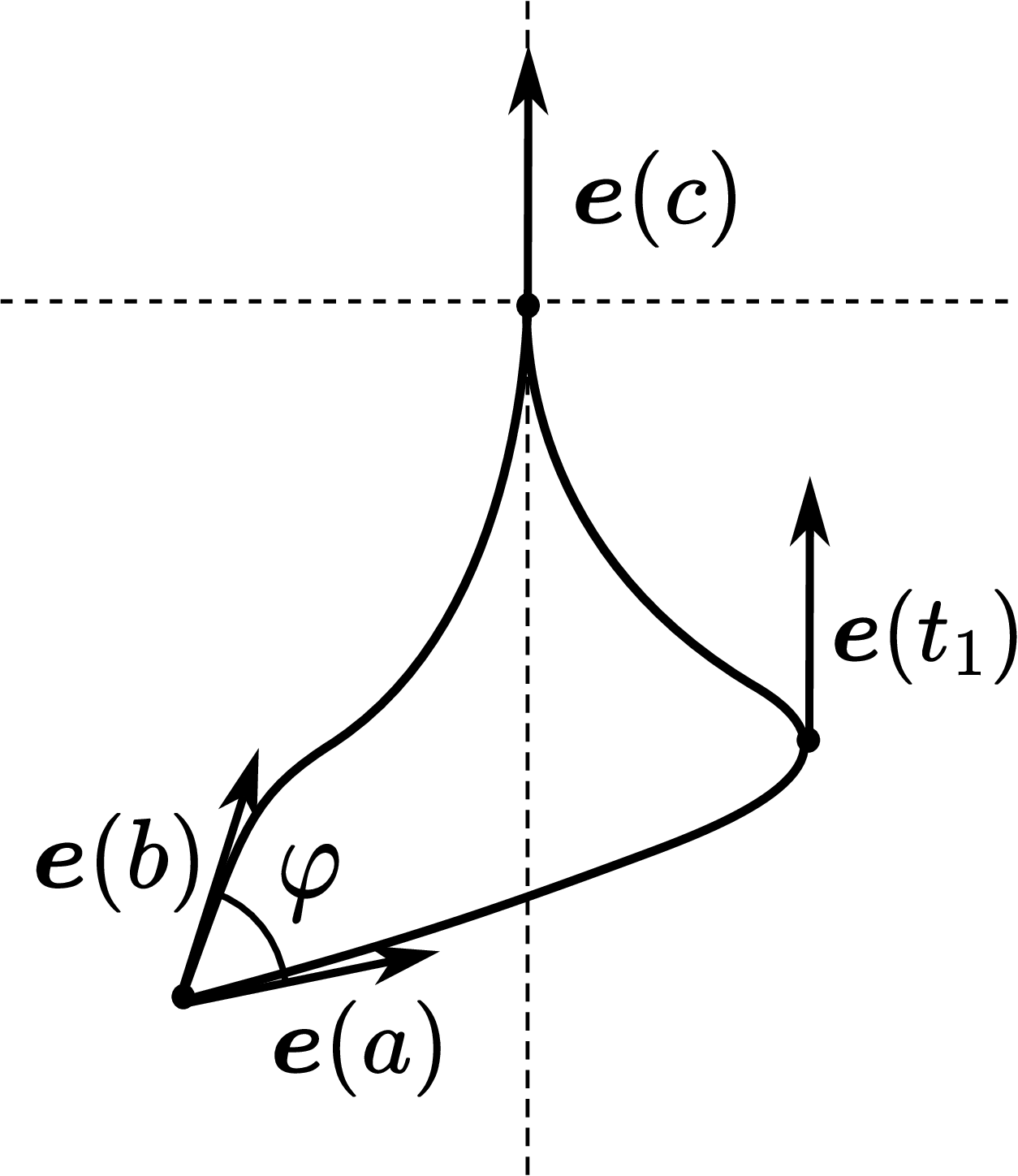}}} &
  \mbox{\raisebox{1.4mm}{\includegraphics[width=30mm]{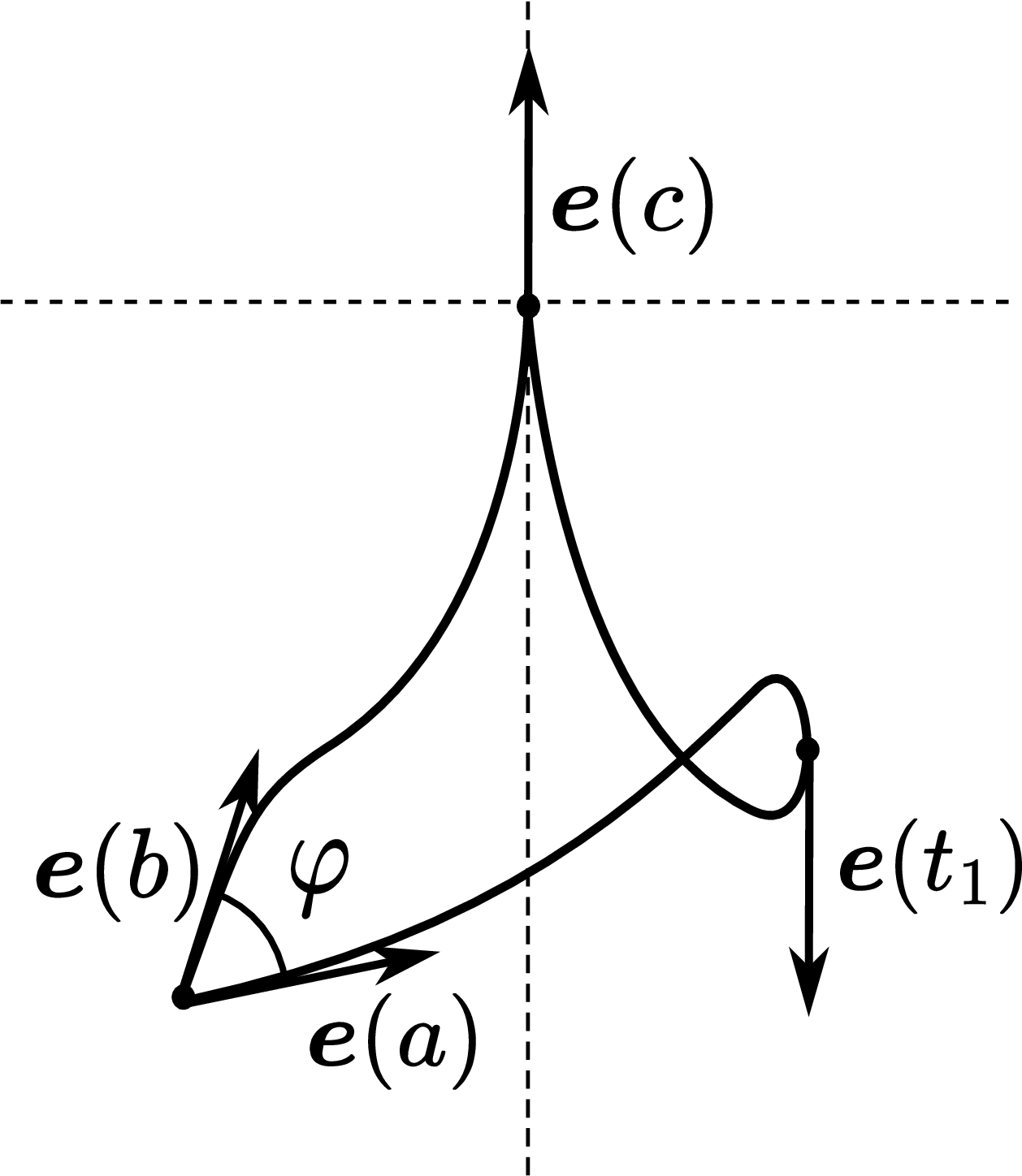}}} &
  \mbox{\raisebox{0mm}{\includegraphics[width=32mm]{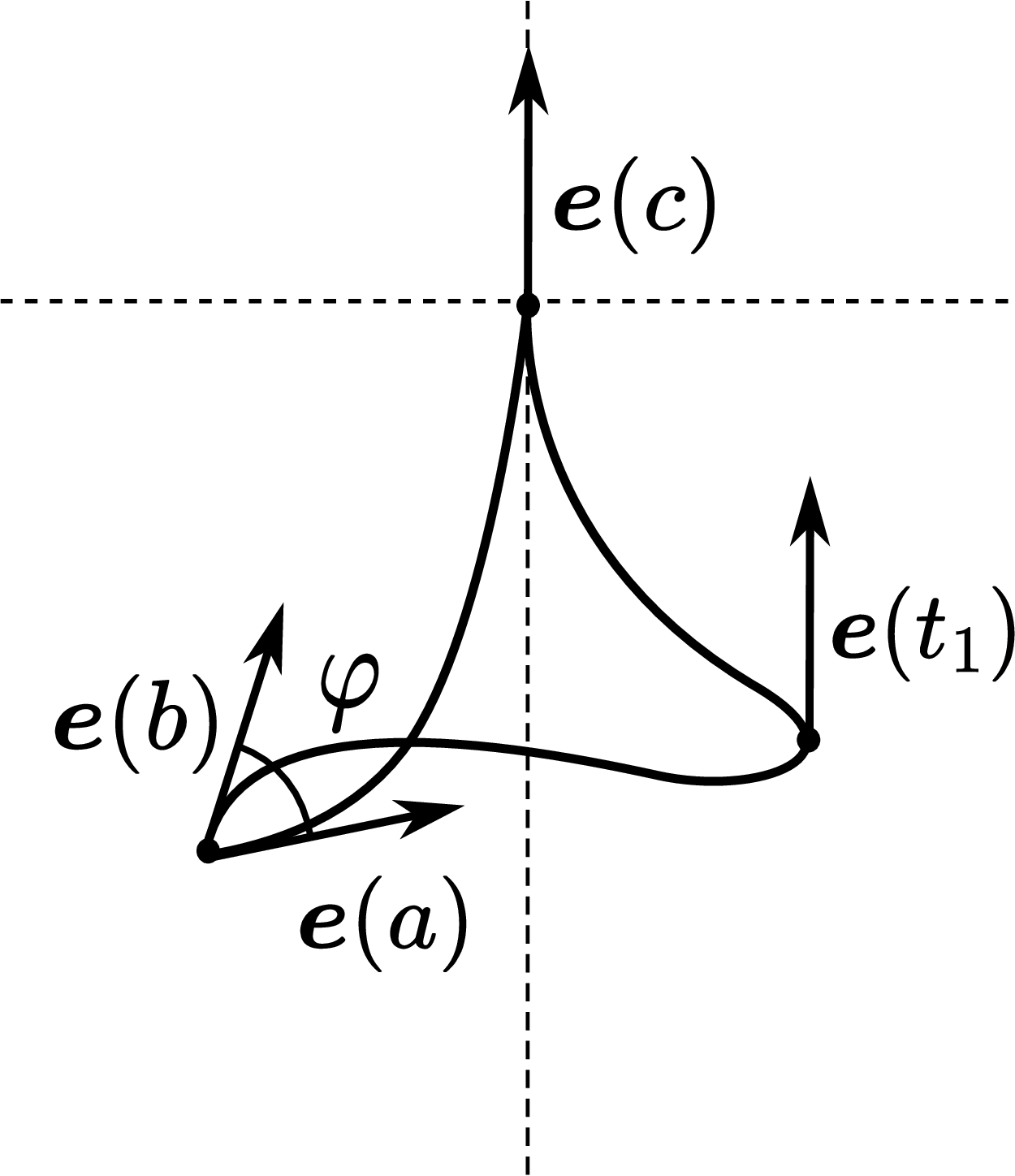}}} &
  \mbox{\raisebox{0.5mm}{\includegraphics[width=32mm]{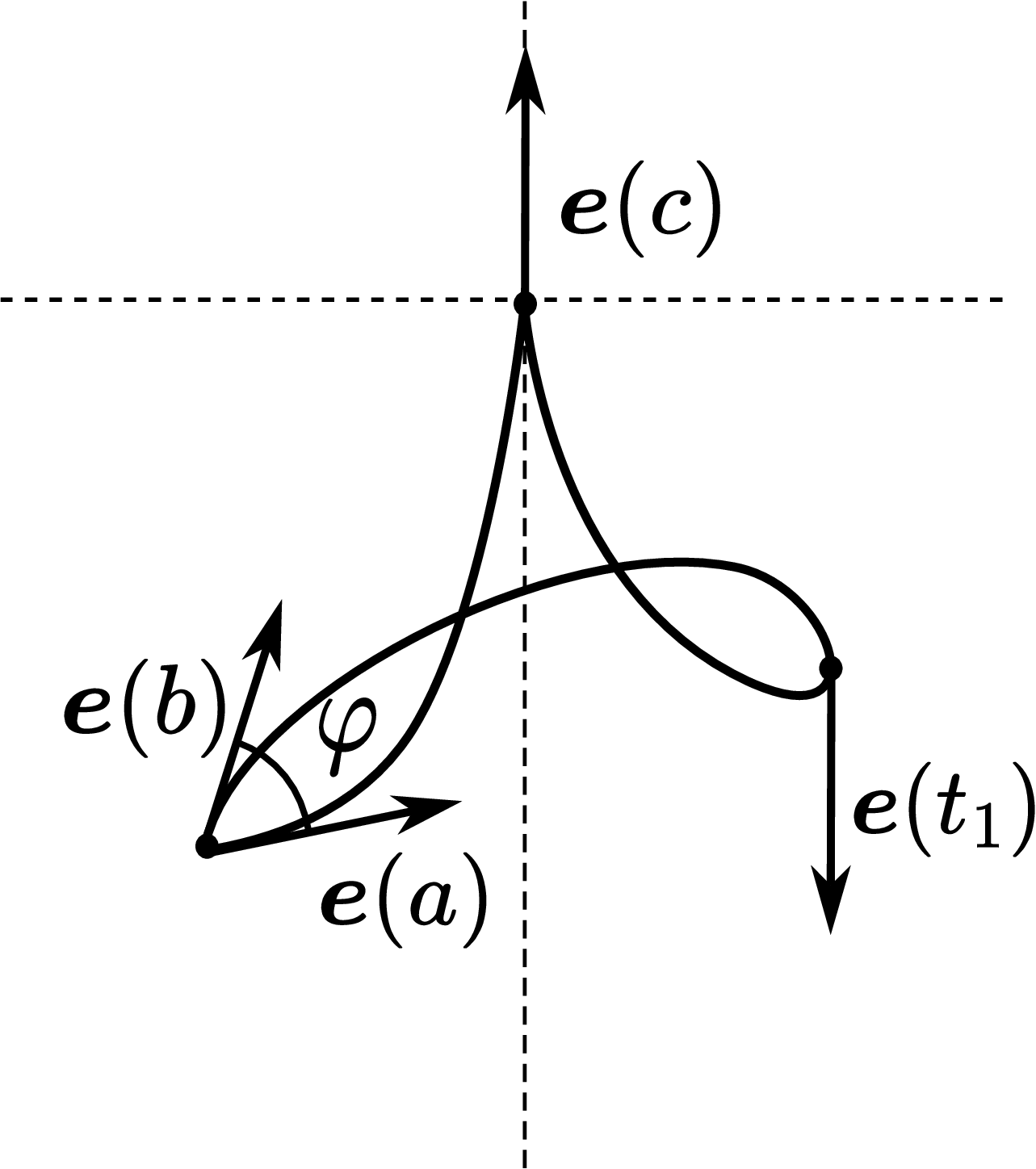}}} 
  \vspace{6mm}\\
  {\footnotesize Case (V-i)} &
  {\footnotesize Case (V-ii)} &
  {\footnotesize Case (VI-i)} &
  {\footnotesize Case (VI-ii)} 
 \end{tabular}  
 \caption{
 Frontals having the same endpoints but with one singular point in the interior,
 the cases (V) and (VI).
 }
 \label{fig:caseV}
\end{figure}

\medskip

By the cases (I)--(VI), we have $K(\gamma)>\varphi$.
\end{proof}

%%%%%%%%%%%%%%%%%%%%%%%%%%%%%%
%%%%%%%%%%%%%%%%%%%%%%%%%%%%%%
%%%%%%%%%%%%%%%%%%%%%%%%%%%%%%
%%%%%%%%%%%%%%%%%%%%%%%%%%%%%%
\subsection{Simpleness and the number of singular points}
\label{sec:proof-B}

Using the estimates of the total absolute curvature
in Fact \ref{fact:sing-0} and Propositions \ref{prop:cusp-2},  \ref{prop:cusp-1},
we prove the following.

\begin{theorem}\label{thm:main}
Let $\gamma : S^1 \to \R^2$ 
be a non-co-orientable closed front 
satisfying $K(\gamma)=\pi$.
If every singular point of $\gamma$ is cusp,
and the number of the singular points of $\ga$ is $3$, 
then $\ga$ is a simple closed curve.
\end{theorem}

\begin{proof}
Suppose $t=0,c_1,c_2$ $(0<c_1<c_2<2\pi)$ are cusps.
We prove that $\gamma : [0,2\pi) \to \R^2$ is injective by contradiction.
So, we assume that there exists $a,b \in [0,2\pi)$ $(a<b)$ satisfies 
$\gamma(a)=\gamma(b).$
By a translation of $t$, if necessary, we may 
suppose that $0\le a< c_1$.
Either one of the cases (1), (2) or (3) occurs,
cf.\ Figure \ref{fig:proof-main}.

\begin{itemize}
\setlength{\leftskip}{-4mm}
\item[(1)]
In the case of $a< b\leq c_1$, 
since the restriction $\gamma|_{[a,b]}$ satisfies 
the assumption of Fact \ref{fact:sing-0}, we have 
$K(\gamma)\geq \int_a^b |\kappa|\, ds>\pi.$
Similarly, in the cases of 
$a=0$, $c_2< b\leq 2\pi$, 
or $0< a< c_1$, $c_2< b\leq 2\pi$, 
the restriction $\gamma|_{[b,2\pi+a]}$ satisfies 
the assumption of Fact \ref{fact:sing-0},
and hence, we have $K(\gamma)>\pi.$
See Figure \ref{fig:proof-main} (1).
\item[(2)]
In the case of $0<a<c_1$ and $c_1<b<c_2$,
We set $\gamma_1=\gamma|_{[a,b]}$ and
$\gamma_2=\gamma|_{[b,2\pi +a]}$.
We let $\varphi\in[0,\pi]$ be the angle at the endpoints of $\gamma_1$.
By Proposition \ref{prop:cusp-1}, we have $K(\gamma_1)>\varphi$.
On the other hand, the angle at the endpoints of $\gamma_2$
is written as $\varphi.$
Thus, by Proposition \ref{prop:cusp-2}, 
we have $K(\gamma_2)\geq\pi-\varphi$.
Hence, we obtain $K(\gamma)=K(\gamma_1)+K(\gamma_2)>\pi$.
The case of $0<a<c_1$ and $c_2<b<2\pi$
is proved in the same way.
See Figure \ref{fig:proof-main} (2).
\item[(3)]
In the case of $a=0$ and $c_1<b<c_2$,
we set $\gamma_1=\gamma|_{[a,b]}$ and
$\gamma_2=\gamma|_{[b,2\pi]}$.
We let $\varphi\in[0,\pi]$ be the angle at the endpoints of $\gamma_1$.
Then $\varphi=\arccos(-T(a)\cdot T(b))$.
By Proposition \ref{prop:cusp-1}, we have $K(\gamma_1)>\varphi$.
On the other hand, the angle at the endpoints of $\gamma_2$
is written as
\begin{align*}
\arccos(-T(b)\cdot T(2\pi))
=\arccos(T(a)\cdot T(b))
%&=\pi-\arccos(-T(a)\cdot T(b))\\
=\pi-\varphi.
\end{align*}
Thus, by Proposition \ref{prop:cusp-1}, 
we have $K(\gamma_2)>\pi-\varphi$.
Hence, we obtain $K(\gamma)=K(\gamma_1)+K(\gamma_2)>\pi$.
The case of $0<a<c_1$ and $b=c_2$
is proved in the same way,
by setting 
$\gamma_1=\gamma|_{[a,b]}$,
$\gamma_2=\gamma|_{[b,2\pi+a]}$.
See Figure \ref{fig:proof-main} (3).
\end{itemize}

\begin{figure}[htbp]
\centering
 \begin{tabular}{c@{\hspace{10mm}}c@{\hspace{14mm}}c}
 %{ccc}
  \resizebox{3.3cm}{!}{\includegraphics{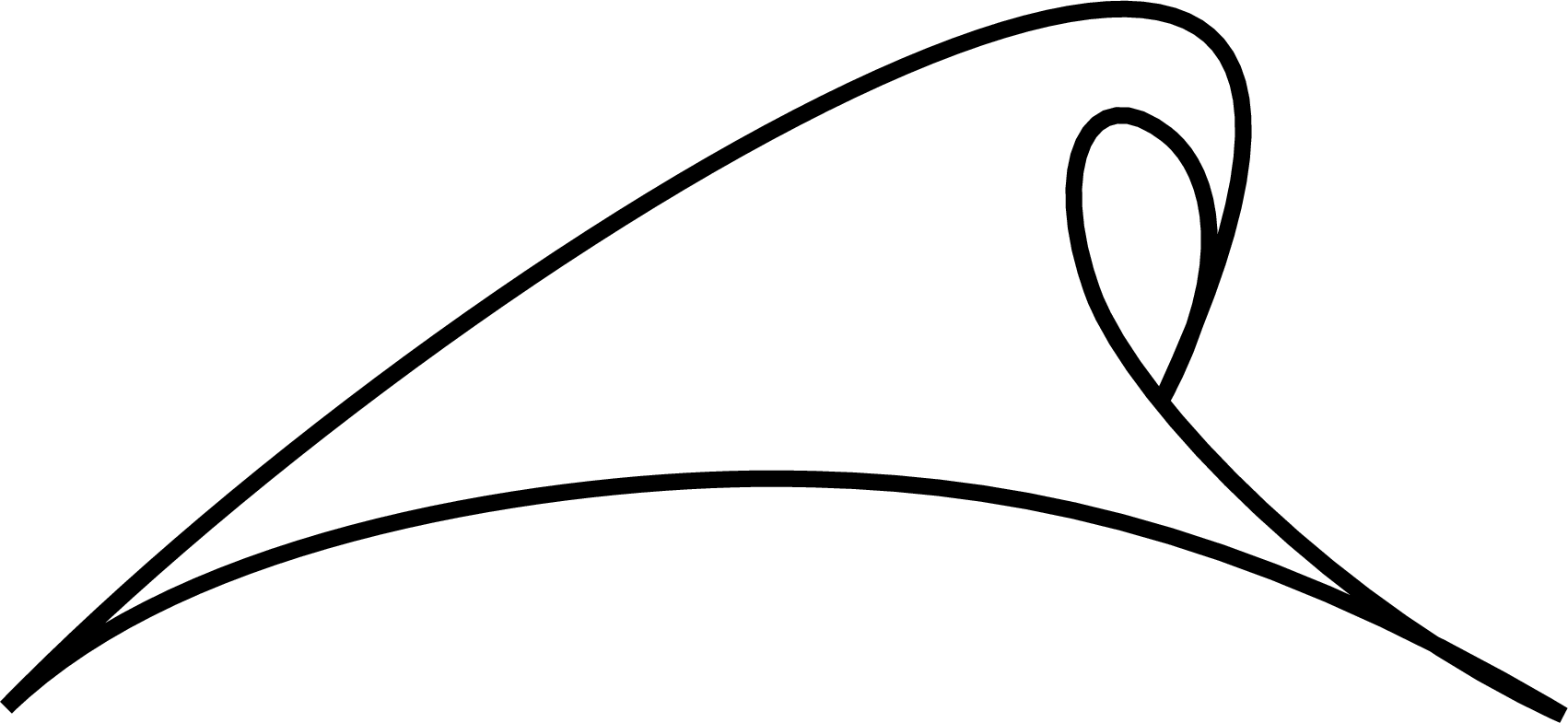}} &
  \resizebox{3.3cm}{!}{\includegraphics{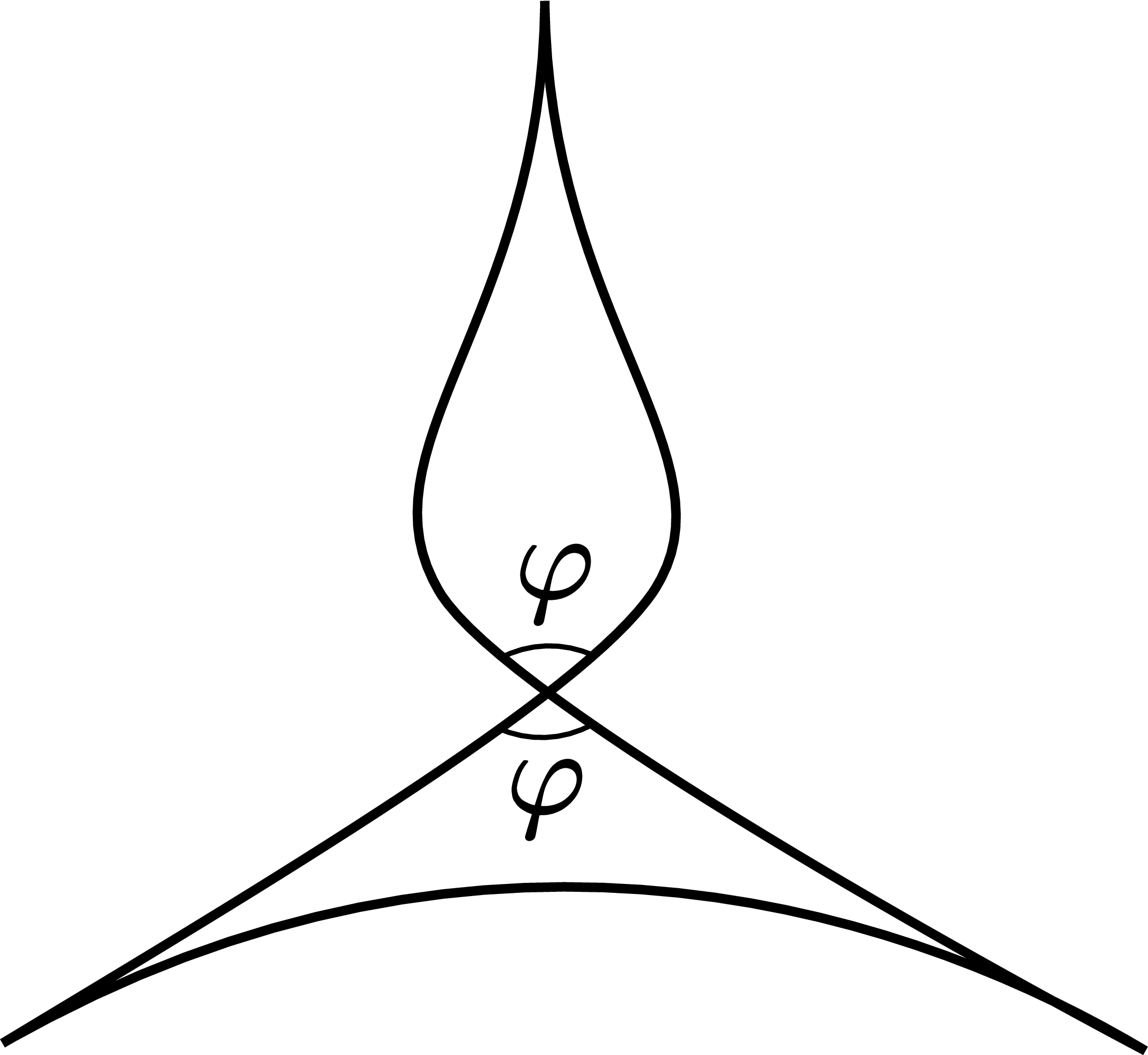}} &
  \resizebox{3.75cm}{!}{\includegraphics{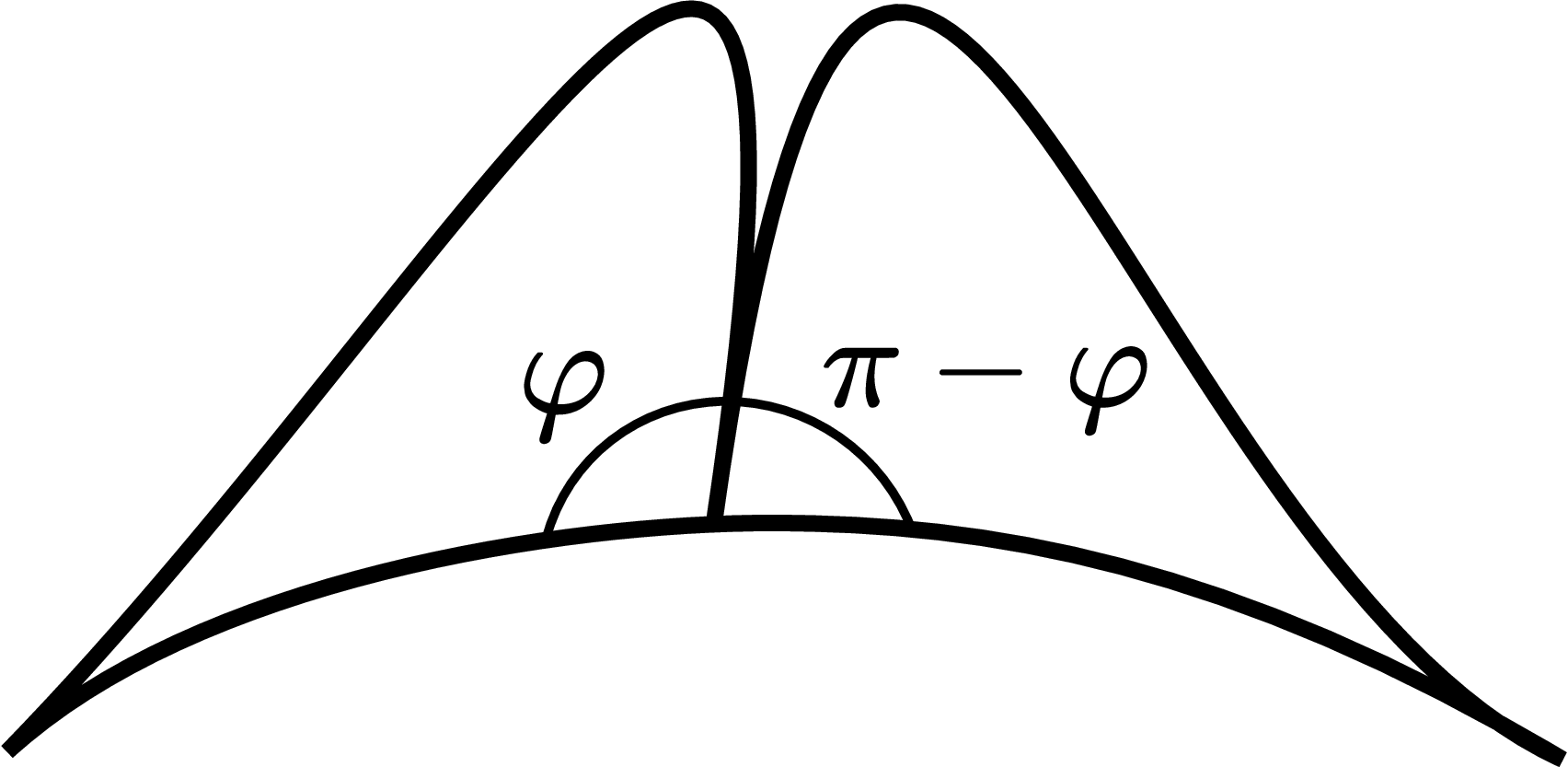}}\\
  {\footnotesize (1)} &
  {\footnotesize (2)} &
  {\footnotesize \hspace{-10mm} (3)} 
 \end{tabular}  
 \caption{
 Possible curve shapes when having self-intersections.
 }
 \label{fig:proof-main}
\end{figure} 

\medskip
Therefore, by the cases (1), (2) and (3), 
we have $K(\gamma)>\pi$, a contradiction.
\end{proof}

\begin{proof}[Proof of Theorem \ref{thm:intro3}]
Let $\gamma : S^1\to \R^2$ be a
non-co-orientable closed front
with $K(\gamma)=\pi$.
Suppose that every singular point is a cusp.
By the non-co-orientability of $\gamma$,
the number $\M$ of cusps is 
an odd integer (for example, see \cite{SUY-book}).
If $\M=1$, then Fact \ref{fact:sing-0}
yields that $K(\gamma)>\pi$,
a contradiction.
Hence, we have $\M\geq3$.

By Theorem \ref{thm:main},
it suffices to prove that
if $\gamma : S^1\to \R^2$ is simple then $\M=3$.
Denote by $\Omega$ the interior domain of $\gamma$.
Changing the orientation of $\gamma(t)$, 
if necessary, we may assume that
the left-hand side of $\gamma(t)$ is the interior domain $\Omega$.
Let $c_j\in [0,2\pi)$ $(j=1,\dots, \M)$ be the singular points
of $\gamma(t)$,
and $\angle C_j$ 
be the interior angle of $\gamma(t)$
at $C_j=\gamma(c_j)$ for $j=1,\dots, \M$.
Then, by the Gauss-Bonnet formula,
we have
\begin{equation}\label{eq:GB}
 \sum_{j=1}^{\M}\angle C_j  
 = (\M-2)\pi + \int_{\partial \Omega} \kappa\,ds.
\end{equation}
Since every singular point of $\gamma(t)$ 
is cusp, we have
$\angle C_j=0$ for $j=1,\dots, \M$.
Moreover, as $\gamma(t)$ is locally $L$-convex,
we have $\tilde{\kappa}(t)\geq 0$ or 
$\tilde{\kappa}(t)\leq 0$.
Thus, we may set $\tilde{\kappa}(t)=\sigma|\tilde{\kappa}(t)|$
where $\sigma\in \{+,-\}$.
Then
$$
 \int_{\partial \Omega} \kappa\,ds
 = \int_{S^1} \tilde{\kappa}\,dt
 =\sigma \int_{S^1} |\tilde{\kappa}|\,dt
 =\sigma \pi
$$
holds.
Then,
\eqref{eq:GB} yields that 
$0= (\M-2)\pi + \sigma\pi.$
If $\sigma=+$, we have $\M=1$.
By Fact \ref{fact:sing-0},
$K(\gamma)>\pi$, a contradiction.
Hence, we have $\sigma=-$,
and $\M=3$ holds.
\end{proof}

%%%%%%%%%%%%%%%%%%%%%%%%%%%%%%%%%%
%%%%%%%%%%%%%%%%%%%%%%%%%%%%%%%%%%
%%%%%%%%%%%%%%%%%%%%%%%%%%%%%%%%%%
%%%%%%%%%%%%%%%%%%%%%%%%%%%%%%%%%%
\section{Examples}
\label{sec:examples}

If a frontal $\gamma$ admits a cusp singularity,
then $\kappa$ is unbounded,
and hence, $K(\gamma)>0$ holds.
However, there are examples of a closed front having cusps 
whose total absolute curvature $K(\gamma)$ 
is less than a given positive number.

\begin{example}[Eye-shaped closed frontal]
\label{ex:megata}
For a positive number $a$, we set 
\begin{equation}\label{eq:mgt}
\gamma_a(t)=\frac{1}{5-3\cos2t}
\left(3\cos t-\cos3t,~4a\sin^3 t\right).
\end{equation}
Then, $\gamma_a(t)$ is a $2\pi$-periodic front with 
unit tangent vector field
$$
\vt(t)
=\frac{1}{\sqrt{G(t,a)}}
\left(-\left(2 \sin ^4 t+\sin ^2t+1\right),a \sin 2t \left(\sin ^2t+1\right)\right),
$$
where we set $H(t)=4\sin^8 t+4\sin^6 t+5\sin^4 t+2\sin^2 t+1$
and 
$
  G(t,a)=H(t)+4a^2\sin^2 t(-\sin^6 t-\sin^4 t+\sin^2 t+1).
$
We remark that $G(t,a)>0$ for each $a>0$ and $t\in \R$.
Since $\vt(t+2\pi)=\vt(t)$ holds for each $t$, 
the front $\gamma_a$ is co-orientable.
Every singular point is given by $t=m\pi$ for an integer $m$.
Since $\gamma_a'(m\pi)=0$ and 
$\det(\gamma_a''(m\pi),\gamma_a'''(m\pi))=-36a$,
every singular point is cusp.

For a given positive number $\varepsilon$,
we set $a=\varepsilon/M$, where
$M=\int^{2\pi}_0{\frac{F(t)}{H(t)}dt}\,(>0)$.
We also set $F(t)=|3\cos2t-1|(5-3\cos2t)$.
Since $G(t,a)\ge H(t)$,
the total absolute curvature $K(\gamma_a)$ satisfies
$$
K(\gamma_a)
=\frac{a}{2} \int^{2\pi}_0 \frac{F(t)}{G(t,a)}\, dt
< \frac{a}{2}\int^{2\pi}_{0} \frac{F(t)}{H(t)}\,dt
<\varepsilon.
$$
Therefore,
$\gamma_a$ is a closed co-orientable front 
whose singular set consists of cusps,
and its total absolute curvature $K(\gamma_a)$ 
is less than a given positive number.
\end{example}

\begin{example}[Hypocycloid]
\label{ex:hypo} 
Let $m$ be a positive integer.
We set 
$\gamma : \R\to \R^2$ as
$$
\gamma(t)
=\left(
m \cos (m+1)t+(m+1)\cos mt,~m \sin (m+1)t-(m+1)\sin mt
\right).
$$
For each singular point $t=c$,
it holds that $\gamma'(c)=0$
and $\det(\gamma''(c),\gamma'''(c))\ne0$.
Every singular point is a cusp.
Hence, $\gamma$ is a closed front with period $2\pi$.
%$$
%\gamma'(t)=-\frac{2a(m-n)}{m}\sin\frac{m}{2n}t \, \vt(t),\qquad
%|\gamma'(t)|=\frac{2a(m-n)}{m}\left|\sin\frac{m}{2n}t\right|.
%$$
%
Since $\vt(t)=(\cos (t/2),~\sin (t/2))$ gives
a unit tangent vector field
satisfying $\vt(t+2\pi)=-\vt(t)$,
$\gamma$ is non-co-orientable. 
Then the total absolute curvature $K(\gamma)$ is given by
%$$
%\det(\gamma'(t),\gamma''(t))
%=\frac{2a^2(m-n)^2 (-m+2n)}{m^2 n}\sin^2\frac{m}{2n}t,
%$$
%$\kappa$ can be written by:
%$$
%\kappa(t)=\frac{m(-m+2n)}{4an(m-n)|\sin\frac{m}{2n}t|}.
%$$
$$
K(\gamma)
=\int^{2\pi}_0|\kappa(t)|\|\gamma'(t)\|\,dt
=\pi.
$$
For the figure of these hypocycloids,
see Figure \ref{fig:curve}.
Figure \ref{fig:curve} (a), (b) and (c) 
are the figures of the hypocycloids of 
$m=1, 2$ and $3$,
respectively.
In these cases, the total absolute curvature $K(\gamma)$ is $\pi$.
\end{example}

%%%%%%%%%%%%%%%%%%%%%%%%%%%%%%%%%%
%%%%%%%%%%%%%%%%%%%%%%%%%%%%%%%%%%
%%%%%%%%%%%%%%%%%%%%%%%%%%%%%%%%%%
%%%%%%%%%%%%%%%%%%%%%%%%%%%%%%%%%%
\begin{acknowledgements}
The authors 
%expresses gratitude to 
%Wayne Rossman for careful reading of the first draft.
%He also 
would like to thank %Kentaro Saji, Keisuke Teramoto,
Masaaki Umehara %and Kotaro Yamada 
for helpful comments.
\end{acknowledgements}

%%%%%%%%%%%%%%%%%%%%%%%%%%%%%%%%%%
%%%%%%%%%%%%%%%%%%%%%%%%%%%%%%%%%%
%%%%%%%%%%%%%%%%%%%%%%%%%%%%%%%%%%
%%%%%%%%%%%%%%%%%%%%%%%%%%%%%%%%%%

\end{document}